\documentclass[12pt]{amsart}
\usepackage[T1]{fontenc}
\usepackage[utf8]{inputenc}
\usepackage[english]{babel}
\usepackage{amsmath}
\usepackage{amssymb}
\usepackage{amsthm}

\usepackage{MnSymbol}
\usepackage{wasysym}
\usepackage{euscript}

\usepackage[babel=true]{csquotes}

\usepackage[marginparwidth=2cm]{geometry}
\usepackage[usenames,dvipsnames,svgnames,table]{xcolor}
\usepackage[colorlinks=true,urlcolor=DarkBlue,linkcolor=DarkRed,citecolor=DarkGreen,pagebackref=true]{hyperref}

\usepackage{tikz}
\usetikzlibrary{shapes,arrows}
\usetikzlibrary{fit,positioning}
\usetikzlibrary{patterns}
\usepackage{xkeyval}
\usepackage{moreverb}
\usepackage{epic}
\usepgfmodule{shapes,plot,decorations}

\usepackage{etex}
\usepackage{stmaryrd}

\usepackage{multicol}
\usepackage[normalem]{ulem}

\usepackage{pinlabel}
\usepackage{subfig}
\usepackage[switch*]{lineno}
\usepackage{tikz}
\usetikzlibrary{matrix,arrows,decorations.pathmorphing}

\newcommand{\N}{\mathbb{N}}
\newcommand{\Z}{\mathbb{Z}}

\newcommand{\R}{\mathbb{R}}
\newcommand{\Cb}{\mathbb{C}}

\newcommand{\Hb}{\mathbb{H}}
\newcommand{\Gb}{\mathbb{G}}

\newcommand{\A}{\mathbb{A}}
\renewcommand{\S}{\mathbb{S}}
\newcommand{\Pb}{\mathbb{P}}
\newcommand{\Pd}{\mathbb{P}^d}

\newcommand{\Eb}{\mathbb{E}}

\newcommand*{\so}[1]{\mathrm{PSO}_{#1,1}}
\newcommand*{\s}[1]{\mathrm{SL}_{#1}(\mathbb{R})}

\renewcommand*{\sl}[1]{\mathfrak{sl}_{#1}(\mathbb{R})}
\newcommand{\lieo}{\mathfrak{o}}
\newcommand*{\lieso}[1]{\mathfrak{so}_{#1,1}}
\newcommand{\Idd}{\mathrm{Id}}

\renewcommand{\gg}{\mathfrak{g}}

\newcommand{\G}{\Gamma}
\newcommand{\g}{\gamma}

\renewcommand{\O}{\Omega}

\newcommand{\U}{\mathcal{U}}

\newcommand{\V}{\mathcal{V}}

\newcommand{\E}{\mathcal{E}}

\newcommand{\Gc}{\mathcal{G}}

\newcommand{\Cc}{\mathcal{C}}

\newcommand{\Aut}{\textrm{Aut}}

\newcommand{\Aff}{\textrm{Aff}}
\newcommand{\Isom}{\textrm{Isom}}

\newcommand{\Hom}{\mathrm{Hom}}

\newcommand{\ra}{\rightarrow}
\newcommand{\emp}{\emptyset}
\newcommand{\orb}{\mathcal{O}}
\newcommand{\torb}{\widetilde{\mathcal{O}}}
\newcommand{\PGL}{\mathrm{PGL}}
\newcommand{\SL}{\mathrm{SL}}
\newcommand{\SI}{\mathbb{S}}

\newcommand{\Def}{\mathrm{Def}}

\newcommand{\hP}{\hat{P}}

\renewcommand{\leq}{\leqslant}
\renewcommand{\geq}{\geqslant}
\renewcommand{\epsilon}{\varepsilon}

\usepackage{aurical}
\newcommand{\B}{\textrm{\Fontauri C\normalfont}}
\newcommand{\C}{\B}

\theoremstyle{plain}
\newtheorem{theorem}{Theorem}[section]

\newtheorem{propo}[theorem]{Proposition}

\newtheorem{lemma}[theorem]{Lemma}

\theoremstyle{definition}
\newtheorem*{ex}{Example}
\newtheorem{de}[theorem]{Definition}

\theoremstyle{remark}
\newtheorem{rem}[theorem]{Remark}

\makeindex

\title[Deformations of convex real projective manifolds and orbifolds]{Deformations of convex real projective \\ manifolds and orbifolds}

\author{
Suhyoung Choi
}
\address{Department of Mathematical Sciences, KAIST,  Republic of Korea}
\email{schoi@math.kaist.ac.kr}

\author{
Gye-Seon Lee
}
\address{Mathematisches Institut, Ruprecht-Karls-Universit\"{a}t Heidelberg, Germany}
\email{lee@mathi.uni-heidelberg.de}

\author{
Ludovic Marquis
}
\address{IRMAR, University of Rennes, France}
\email{ludovic.marquis@univ-rennes1.fr}

\begin{document}
\let\oldmarginpar\marginpar
\renewcommand\marginpar[1]{\-\oldmarginpar[\raggedleft\tiny #1]%
{\raggedright\tiny #1}}

\begin{abstract}
In this survey, we study representations of finitely generated groups into Lie groups, focusing on the deformation spaces of convex real projective structures on closed manifolds and orbifolds, with an excursion on projective structures on surfaces. We survey the basics of the theory of character varieties, geometric structures on orbifolds, and Hilbert geometry. The main examples of finitely generated groups for us will be Fuchsian groups, 3-manifold groups and Coxeter groups.
\end{abstract}

\thanks{
S. Choi was supported by the Mid-career Researcher Program through the NRF grant NRF-2013R1A1A2056698 funded by the MEST.
G.-S. Lee was supported by the DFG research grant “Higher Teichm\"{u}ller Theory” and he acknowledges support from the U.S. National Science Foundation grants DMS 1107452, 1107263, 1107367 “RNMS: GEometric structures And Representation varieties” (the GEAR Network). L. Marquis acknowledges support from the French ANR programs Finsler and Facets.  The authors thank Lizhen Ji, Athanase Papadopoulos and Shing-Tung Yau for the opportunity to publish this survey in “Handbook of Group Actions”. Finally, we would like to thank the referee and Athanase Papadopoulos for carefully reading this paper and suggesting several improvements.
}

\maketitle

\tableofcontents

\section{Introduction}

\par{
The goal of this paper is to survey the deformation theory of convex real projective structures on manifolds and orbifolds. Some may prefer to speak of discrete subgroups of the group $\PGL_{d+1}(\R)$ of projective transformations of the real projective space $\R\Pb^d$ which preserve a properly convex open subset of $\R\Pb^d$, and some others prefer to speak of Hilbert geometries. 
\\
\par{
Some motivations for studying this object are the following:
}

\subsection{Hitchin representations}

\par{
Let $S$ be a closed surface of genus $g \geqslant 2$ and let $\G$ be the fundamental group of $S$. There is a unique irreducible representation $i_m : \mathrm{PSL}_2(\R) \to \mathrm{PSL}_m(\R)$, up to conjugation. A representation $\rho:\G \to \mathrm{PSL}_m(\R)$ is called a {\em Hitchin representation}\index{Hitchin representation}\index{representation!Hitchin} if there are a continuous path $\rho_t \in \textrm{Hom}(\G,\mathrm{PSL}_m(\R))$ and a discrete, faithful representation $\tau : \G \to \mathrm{PSL}_2(\R)$ such that $\rho_0 = \rho$ and $\rho_1 = i_m \circ \tau :\G \to \mathrm{PSL}_m(\R)$. The space $\mathrm{Hit}_m(\G)$ of conjugacy classes of Hitchin representations of $\G$ in $\mathrm{PSL}_m(\R)$ has a lot of interesting properties: Each connected component is homeomorphic to an open ball of dimension $2(g-1)(m^2-1)$ (see Hitchin \cite{hitchin}), and every Hitchin representation is discrete, faithful, irreducible and purely loxodromic (see Labourie \cite{labourie}). 
}
\\
\par{
When $m=3$, the first author and Goldman \cite{Choi_gold_close} showed that each Hitchin representation preserves a properly convex domain\footnote{We abbreviate a connected open set to a {\em domain}\index{domain}.} of $\R\Pb^2$. In other words,  $\mathrm{Hit}_3(\G)$ is the space $\C(S)$ of marked convex real projective structures on the surface $S$.
}
\\
\par{To understand the geometric properties of Hitchin representations, Labourie \cite{labourie} introduced the concept of Anosov representation. Later on, Guichard and Wienhard \cite{Guichard} studied this notion for finitely generated Gromov-hyperbolic groups. For example, if $M$ is a closed manifold whose fundamental group is Gromov-hyperbolic, then the holonomy representations of convex real projective structures on $M$ are Anosov.
}

\subsection{Deformations of hyperbolic structures}

\par{
Let $M$ be a closed hyperbolic manifold of dimension $d \geqslant 3$, and let $\G=\pi_1(M)$. By Mostow rigidity, up to conjugation, there is a unique faithful and discrete representation $\rho_{\mathrm{geo}}$ of $\G$ in $\so{d}(\R)$. The group $\so{d}(\R)$ is canonically embedded in $\PGL_{d+1}(\R)$. We use the same notation $\rho_{\mathrm{geo}}:\G \to \PGL_{d+1}(\R)$ to denote the composition of $\rho_{\mathrm{geo}} : \G \to \so{d}(\R)$ with the canonical inclusion. Now, there is no reason that $\rho_{\mathrm{geo}}$ is the unique faithful and discrete representation of $\G$ in $\PGL_{d+1}(\R)$, up to conjugation.
}
\\
\par{
In fact, there are examples of closed hyperbolic manifold $M$ of dimension $d$ such that $\G$ admits discrete and faithful representations in $\PGL_{d+1}(\R)$ which are not conjugate to $\rho_{\mathrm{geo}}$ (see Theorem \ref{thm:JM}). We can start looking at the connected component $\mathrm{Ben}(M)$ of the space of representations of $\G$ into $\PGL_{d+1}(\R)$ containing $\rho_{\mathrm{geo}}$, up to conjugation. The combination of a theorem of Koszul and a theorem of Benoist implies that every representation in $\mathrm{Ben}(M)$ is discrete, faithful, irreducible and preserves a properly convex domain $\O$ of $\R\Pb^d$.\footnote{The action of $\G$ on $\O$ is automatically proper and cocompact for general reasons.}
}  
\\
\par{
At the moment of writing this survey, there is no known necessary and sufficient condition on $M$ to decide if $\mathrm{Ben}(M)$ consists of exactly one element, which is the hyperbolic structure. There are infinitely many closed hyperbolic 3-manifolds $M$ such that  $\mathrm{Ben}(M)$ is the singleton (see Heusener-Porti \cite{HP}), and there are infinitely many closed hyperbolic 3-orbifolds $M$ such that $\mathrm{Ben}(M)$ is homeomorphic to an open $k$-ball, for any $k \in \N$ (see Marquis \cite{ecima_ludo}).
}

\subsection{Building blocks for projective surfaces}

\par{
Let $S$ be a closed surface. We might wish to understand all possible real projective structures on $S$, not necessarily only the convex one. The first author showed that convex projective structures are the main building blocks to construct all possible projective structures on the surface $S$ (see Theorem \ref{thm:convdec}). 
}

\subsection{Geometrization}

\par{
Let $\O$ be a properly convex domain of $\mathbb{RP}^d$, and let $\mathrm{Aut}(\O)$ be the subgroup of $\PGL_{d+1}(\R)$ preserving $\O$. There is an $\mathrm{Aut}(\O)$-invariant metric $d_{\O}$ on $\O$, called the {\em Hilbert metric}\index{Hilbert!metric}, that make $(\O,d_{\O})$ a complete proper geodesic metric space, called a {\em Hilbert geometry}\index{Hilbert!geometry}. We will discuss these metrics in Section \ref{hilbert_geo_defi}.  The flavour of the metric space $(\O,d_{\O})$ really depends on the geometry of the boundary of $\O$. For example, on the one hand, the interior of an ellipse equipped with the Hilbert metric is isometric to the hyperbolic plane, forming the projective model of the hyperbolic plane, and on the other hand, the interior of a triangle is isometric to the plane with the norm whose unit ball is the regular hexagon (see de la Harpe \cite{delaharpe}).
}
\\
\par{
Unfortunately, Hilbert geometries are almost never CAT$(0)$: A Hilbert geometry $(\O,d_{\O})$ is  CAT$(0)$ if and only if $\O$ is an ellipsoid (see Kelly-Straus \cite{kelly_straus}). 
However, the idea of Riemmanian geometry of non-positive curvature is a good guide towards the study of the metric properties of Hilbert geometry.
}  
\\
\par{
An irreducible symmetric space $X=G/K$ is “Hilbertizable” if there exist a properly convex domain $\O$ of $\R\Pb^d$ for some $d$ and an irreducible representation $\rho:G \to \mathrm{PSL}_{d+1}(\R)$ such that $\rho(G)$ acts transitively on $\O$ and the stabilizer of a point of $\O$ is conjugate to $K$. The symmetric spaces for $\so{d}(\R)$, $\mathrm{PSL}_m(\mathbb{K})$ for $ \mathbb{K}=\R,\Cb,\Hb$, and the exceptional Lie group $E_{6,-26}$ are exactly the symmetric spaces that are Hilbertizable (see Vinberg \cite{homogeneous_vin_1,homogeneous_vin_2} 
or Koecher \cite{koecher}).
}
\\
\par{
Nevertheless we can ask the following question to start with:
\begin{itemize} 
\item[] “Which manifold or orbifold $M$ can be realized as the quotient of a properly convex domain $\O$ by a discrete subgroup $\G$ of $\mathrm{Aut}(\O)$?”
\end{itemize} 

If this is the case, we say that $M$ {\em admits a properly convex real projective structure}\index{properly convex!real projective structure}\index{geometric structure!real projective structure!properly convex}.
}
\\
\par{
In dimension $2$, the answer is easy: a closed surface $S$ admits 
a convex projective structure if and only if its Euler characteristic is non-positive. The universal cover of a properly convex projective torus is a triangle, and a closed surface of negative Euler characteristic admits a hyperbolic structure, which is an example of a properly convex projective structure.
}
\\
\par{
In dimension greater than or equal to $3$, no definite answer is known; see Section \ref{existence} for a description of our knowledge. To arouse the reader’s curiosity we just mention that there exist manifolds which admit a convex real projective structure but which cannot be locally symmetric spaces. 
}

\subsection{Coxeter groups}

\par{
A Coxeter group is a finitely presented group that “resembles” the groups generated by reflections; see Section \ref{Coxeter} for a precise definition, and de la Harpe \cite{harpe}  for a beautiful invitation.
An important object to study Coxeter group, denoted $W$, is a representation $\rho_{\mathrm{geo}} : W \to \mathrm{GL}(V)$ introduced by Tits \cite{MR0240238}. The representation $\rho_{\mathrm{geo}}$, in fact, preserves a convex domain of the real projective space $\Pb(V)$.
For example, Margulis and Vinberg \cite{margu_vin} used this property of $\rho_{\mathrm{geo}}$ to show that an irreducible Coxeter group is either finite, virtually abelian or large.\footnote{A group is {\em large}\index{large group}\index{Coxeter!group!large} if it contains a subgroup of finite index that admits an onto morphism to a non-abelian free group.}
}
\\
\par{
From our point of view, Coxeter groups are a great source for building groups acting on properly convex domains of $\Pb(V)$. Benoist \cite{CD4,MR2295544} used them to construct the first example of a closed 3-manifold that admits a convex projective structure $\Omega/\Gamma$ such that $\O$ is not strictly convex, or to build the first example of a closed 4-manifold that admits a convex projective structure $\Omega/\Gamma$ such that $(\O, d_{\O})$ is Gromov-hyperbolic but not quasi-isometric to the hyperbolic space (see Section \ref{existence}).
}

\section{Character varieties}

All along this article, we study the following kind of objects:
\begin{itemize}
\item a finitely generated group $\Gamma$ which we think of as the fundamental group of a complete real hyperbolic manifold/orbifold or its siblings,

\item a Lie group $G$ which is also the set of real points of an algebraic group $\Gb$, and

\item a real algebraic set $\textrm{Hom}(\Gamma,G)$.
\end{itemize}

We want to understand the space $\textrm{Hom}(\Gamma,G)$. First, the group $G$ acts on $\textrm{Hom}(\Gamma,G)$ by conjugation. We can notice that the quotient space is not necessarily Hausdorff since the orbit of the action of $G$ on $\textrm{Hom}(\Gamma,G)$ may not be closed. But the situation is not bad since each orbit closure contains at most one closed orbit. Hence, a solution to the problem is to forget the representations whose orbits are not closed. Let us recall the characterization of the closedness of the orbit:

\begin{lemma}[Richardson \cite{Richardson}] \label{lem:Richardson} 
Assume that $G$ is the set of real points of a reductive\footnote{An algebraic group is {\em reductive}\index{algebraic group!reductive}\index{reductive!algebraic group} if its unipotent radical is trivial.}  algebraic group defined over $\R$. Let $\rho:\Gamma \to G$ be a representation. Then the orbit $G \cdot \rho$ is closed if and only if the Zariski closure of $\rho(\Gamma)$ is a reductive subgroup of $G$. Such a representation is called a {\em reductive} representation\index{reductive!representation}\index{representation!reductive}.
\end{lemma}

Define
$$
R(\Gamma,G) = \textrm{Hom}(\Gamma,G)/G
\quad \mathrm{and} \quad
\chi(\Gamma,G) = \{ [\rho] \in \textrm{R}(\Gamma,G) \, |\, \rho \textrm{ is reductive} \}.
$$
These spaces are given with the quotient topology and the subspace topology, respectively. 

\begin{theorem}[Topological, geometric and algebraic viewpoint, Luna \cite{Luna1,Luna2} and Richardson-Slodowy \cite{richardson_slodowy}]\label{th:structure_character_variety}
Assume that $G$ and $\rho$ are as in Lemma {\em \ref{lem:Richardson}}. Then
\begin{itemize}
\item There exists a unique reductive representation $\overrightarrow{\rho} \in \overline{G \cdot \rho}$, up to conjugation.

\item The space $\chi(\Gamma,G)$ is Hausdorff and it is identified with the Hausdorff quotient of $R(\Gamma,G)$.

\item The space $\chi(\Gamma,G)$ is also a real semi-algebraic variety which is the {\em GIT}-quotient\index{GIT-quotient}\footnote{GIT is the abbreviation for the Geometric Invariant Theory; see the lecture notes of Brion \cite{Brion} for information on this subject.} of the action of $G$ on $\hbox{{\rm Hom}}(\Gamma,G)$.
\end{itemize}
\end{theorem}

The real semi-algebraic Hausdorff space $\chi(\Gamma,G)$ is called the {\em character variety}\index{character variety} of the pair $(\Gamma,G)$.

\begin{rem}
Theorem \ref{th:structure_character_variety} is not explicitly stated in Luna's work. The statement can be founded in Remark 7.3.3 of Richardson-Slodowy \cite{richardson_slodowy} and is also proved by independent method via the action of the reductive group $G$ on the affine variety  $\hbox{{\rm Hom}}(\Gamma,G)$. See also Section 3 of Bergeron \cite{bergeron_maxime}.
\end{rem}

\subsection*{A baby example}

The space $\chi(\Z,G)$ is the set of semi-simple elements of $G$ modulo conjugation.
\begin{itemize}
\item If $G=\mathrm{SL}_m(\Cb)$, then $\chi(\Z,G)=\Cb^{m-1}$.
\item If $G=\mathrm{SL}_2(\R)$, then $\chi(\Z,G)$ is a circle $\{ e^{i \theta} \,\,|\,\, 0 \leq \theta < 2 \pi\}$ with two half-lines that are glued on the circle at the points $\{ 1\}$ and $\{ -1\}$.
\end{itemize}

\section{Geometric structures on orbifolds}

In this section, we recall the vocabulary of orbifolds and of geometric structures on orbifolds. The reader can skip this section if he or she is  familiar with these notions. A classical reference are Thurston’s lecture notes \cite{Thurston:2002}. See also Goldman \cite{Goldmanexp}, Choi \cite{msjbook}, Boileau-Maillot-Porti \cite{BMP}.
For the theory of orbifolds itself, we suggest the article of Moerdijk-Pronk \cite{Moerd}, the books of Adem-Leida-Ruan \cite{ALR} and of Bridson-Haefliger \cite{BH}. 

\subsection{Orbifolds}

An orbifold\index{orbifold} is a topological space which is locally homeomorphic to the quotient space of $\mathbb{R}^d$ by a finite subgroup of ${\rm Diff}(\mathbb{R}^d)$, the diffeomorphism group of $\mathbb{R}^d$. Here is a formal definition:
A $d$-dimensional {\em orbifold}\index{orbifold} $\mathcal{O}$ consists of a second countable, Hausdorff space $X_{\mathcal{O}}$ 
with the following additional structure:
\begin{enumerate}
\item A collection of open sets $\{ U_i \}_{i \in I}$, for some index $I$, which is a 
covering of $X_{\mathcal{O}}$ and is closed under finite intersections.
\item To each $U_i$ are associated a finite group $\Gamma_i$, a smooth action of $\Gamma_i$ on an open subset $\widetilde{U}_i$ in $\mathbb{R}^d$ and a homeomorphism $\phi_i \,:\, \widetilde{U}_i/\Gamma_i \rightarrow  U_i$.
\item Whenever $U_i \subset U_j$, there are an injective homomorphism $f_{ij} \,:\, \Gamma_i \rightarrow \Gamma_j$ and a
smooth embedding $\widetilde{\phi}_{ij} \,:\, \widetilde{U}_i \rightarrow \widetilde{U}_j$ equivariant with respect to $f_{ij}$, i.e. $\widetilde{\phi}_{ij}(\gamma x) = f_{ij}(\gamma) \widetilde{\phi}_{ij}(x)$ for $\gamma \in \Gamma_i$ and $x \in \widetilde{U}_i$, such that the following diagram commutes:

$$
\begin{tikzpicture}
\matrix(m)[matrix of math nodes,
row sep=2em, column sep=4em,
text height=1.5ex, text depth=0.25ex]
{\widetilde{U}_i & \widetilde{U}_j\\
\widetilde{U}_i/\Gamma_i & \widetilde{U}_j/f_{ij}(\Gamma_i) \\
 & \widetilde{U}_j/\Gamma_j \\
U_i & U_j \\
};
\path[->,font=\small]
(m-1-1) edge node[above]{$\widetilde{\phi}_{ij}$}
 (m-1-2)
(m-1-1) edge (m-2-1)
(m-1-2) edge (m-2-2)
(m-2-1) edge (m-2-2)
(m-2-1) edge node[left]{$\phi_i$} (m-4-1)
(m-2-2) edge (m-3-2)
(m-3-2) edge node[right]{$\phi_j$} (m-4-2)
(m-4-1) edge node[above]{\normalsize $\subset$} (m-4-2);
\end{tikzpicture}
$$

\item The collection $\{ U_i \}$ is maximal relative to the conditions (1) -- (3).
\end{enumerate}

This additional structure is called an {\em orbifold structure}\index{orbifold!structure}, and the space $X_{\orb}$ is the {\em underlying space}\index{orbifold!underlying space} of $\orb$. 
Here, it is somewhat important to realize that $\widetilde{\phi}_{ij}$ is uniquely determined up to compositions of elements of $\Gamma_{i}$ and $\Gamma_{j}$.

\begin{ex}
If $M$ is a smooth manifold and $\Gamma$ is a subgroup of ${\rm Diff}(M)$ acting properly discontinuously on $M$, then the quotient space $M/\Gamma$ has an obvious orbifold structure. 
\end{ex}
 
 \par{
An orbifold is said to be {\em connected}\index{orbifold!connected}, {\em compact}\index{orbifold!compact} or {\em noncompact}\index{orbifold!noncompact} according to whether the underlying space is connected, compact or noncompact, respectively.
}
\\
\par{
A {\em smooth map}\index{orbifold!smooth map} between orbifolds $\mathcal{O}$ and $\mathcal{O}'$ is a continuous map $f : X_{\mathcal{O}} \rightarrow X_{\mathcal{O}' }$ satisfying that for each $x \in \mathcal{O}$ there are coordinate neighborhoods $U \approx \widetilde{U}/\Gamma$ of $x$ in $\mathcal{O}$  and $U' \approx \widetilde{U'}/\Gamma' $ of $f(x)$ in $\mathcal{O'}$ such that $f(U) \subset U'$ and the restriction $f|_{U}$ can be lifted to a smooth map $ \widetilde{f} : \widetilde{U} \rightarrow \widetilde{U}'$ which is equivariant with respect to a homomorphism $\Gamma \rightarrow \Gamma'$. Note that the homomorphism $\Gamma \ra \Gamma'$ may not be injective nor surjective. An {\em orbifold-diffeomorphism}\index{orbifold!diffeomorphism} between $\mathcal{O}$ and $\mathcal{O}'$ is a smooth map $\mathcal{O} \ra \mathcal{O}'$ with a smooth inverse map. If there is an orbifold-diffeomorphism between $\mathcal{O}$ and $\mathcal{O}'$, we denote this by $\mathcal{O} \approx \mathcal{O}'$.
\\
\par{
An orbifold $\mathcal{S}$ is a {\em suborbifold}\index{suborbifold}\index{orbifold!suborbifold} of an orbifold $\mathcal{O}$ if the underlying space $X_{\mathcal{S}}$ of $\mathcal{S}$ is a subset of $X_{\mathcal{O}}$ and for each point $s \in X_{\mathcal{S}}$, there are a coordinate neighborhood  $U = \phi(\widetilde{U}/\Gamma)$ of $s$ in $\mathcal{O}$ and a closed submanifold $\widetilde{V}$ of $\widetilde U$ preserved by $\Gamma$ such that $V = \phi(\widetilde{V}/(\Gamma|_{\widetilde{V}}))$ is a coordinate neighborhood of $s$ in $\mathcal{S}$. Here, since the submanifold $\widetilde{V}$ is preserved by $\G$, we denote by $\Gamma|_{\widetilde{V}}$ the group obtained from the elements of $\G$ by restricting their domains and codomains to $\widetilde{V}$. Note that the restriction map $\Gamma \ra \Gamma|_{\widetilde{V}}$ may not be injective (see also Borzellino-Brunsden \cite{Borz}). For example, let $\sigma_x$ (resp. $\sigma_y$) be the reflection in the $x$-axis $L_x$ (resp. $y$-axis $L_y$) of $\mathbb{R}^2$. Then $L_x/\langle \sigma_y \rangle$ and $L_y/\langle \sigma_x \rangle$ are suborbifolds of $\mathbb{R}^2/\langle \sigma_x, \sigma_y \rangle$, however both maps $\langle \sigma_x, \sigma_y \rangle \ra  \langle \sigma_y \rangle$ and $\langle \sigma_x, \sigma_y \rangle \ra  \langle \sigma_x \rangle$ are not injective.
Our definition is more restrictive than Adem-Leida-Ruan’s \cite{ALR} and
less restrictive than Kapovich’s \cite{Kp}, however, this definition seems to be better for studying decompositions of $2$-orbifolds along $1$-orbifolds. 
}

\subsection{$(G,X)$-orbifolds}

\par{
Let $X$ be a real analytic manifold and let $G$ be a Lie group acting analytically, faithfully and transitively on $X$. An orbifold is a {\em $(G, X)$}-orbifold\index{$(G, X)$-orbifold}\index{orbifold!geometric} if $\Gamma_i$ is a subgroup of $G$, $\widetilde{U}_i$ is an open subset of $X$, and $\widetilde{\phi}_{ij}$ is locally an element of $G$ (c.f. the definition of orbifold). A $(G,X)$-manifold is a $(G,X)$-orbifold with $\Gamma_i$ trivial. A {\em $(G,X)$-structure}\index{$(G,X)$-structure}\index{geometric structure} on an orbifold $\orb$ is an orbifold-diffeomorphism from $\mathcal{O}$ to a $(G,X)$-orbifold $S$.
}
\\
\par{
Here are some examples: Let $\Eb^{d}$ be the $d$-dimensional Euclidean space and let $\Isom(\Eb^{d})$ be the group of isometries of $\Eb^{d}$. Having an $(\Isom(\Eb^{d}),\Eb^{d})$-structure (or {\em Euclidean structure}\index{Euclidean structure}\index{geometric structure!Euclidean structure}) on a manifold $\mathcal{O}$ is equivalent to having a Riemannian metric on $\mathcal{O}$ of sectional curvature zero.  We can also define a {\em spherical structure}\index{spherical structure}\index{geometric structure!spherical structure} or a {\em hyperbolic structure}\index{hyperbolic!structure}\index{geometric structure!hyperbolic structure} on $\mathcal{O}$ and give a similar characterization for each structure.
}
\\
\par{
Let $\A^{d}$ be the $d$-dimensional affine space and let $\Aff(\A^{d})$ be the group of affine transformations, i.e. transformations of the form $x \mapsto Ax + b$ where $A$ is a linear transformation of $\A^{d}$ and $b$ is a vector in $\A^d$. An $(\Aff(\A^{d}), \A^{d})$-structure (or {\em affine structure}\index{affine!structure}\index{geometric structure!affine structure}) on an orbifold $\mathcal{O}$ is equivalent to a flat torsion-free affine connection on $\mathcal{O}$
(see Kobayashi-Nomizu \cite{Kobbook}). Similarly, a $(\PGL_{d+1}(\R),\R \Pb^{d})$-structure (or {\em real projective structure}\index{real projective structure}\index{geometric structure!real projective structure}) on $\mathcal{O}$ is equivalent to a projectively flat torsion-free affine connection on $\mathcal{O}$ (see Eisenhart \cite{Eis}).
}

\subsection{A tool kit for orbifolds}\label{subsection:kit}

\par{
To each point $x$ in an orbifold $\mathcal{O}$ is associated a group $\Gamma_x$  called the {\em isotropy group}\index{isotropy group}\index{orbifold!isotropy group} of $x$: In a local coordinate system $U \approx \widetilde{U}/\G$ this is the isomorphism class of the stabilizer $\Gamma_{\tilde x} \, \leqslant \G$ of any inverse point $\tilde x$ of $x$ in $\widetilde{U}$. The set $\{ x \in X_\mathcal{O} \,|\, \Gamma_x \neq  \{ 1 \}  \}$ is the {\em singular locus}\index{singular!locus}\index{orbifold!singular locus} of $\mathcal{O}$.
}
\\
\par{
In general, the underlying space of an orbifold is not even a manifold. However, in dimension two, it is homeomorphic to a surface with/without boundary. Moreover, the singular locus of a 2-orbifold can be classified into three families because there are only three types of 
finite subgroups in the orthogonal group ${\mathrm O}_2(\mathbb{R})$\,: 
}

\begin{itemize}
\item {\em Mirror}\index{mirror}\index{orbifold!singular locus!mirror}: $\R^{2}/\Z_{2}$ when $\Z_{2}$ acts by reflection.
\item {\em Cone points of order $n \geqslant 2$}\index{cone point}\index{orbifold!singular locus!cone point}: $\R^{2}/\Z_{n}$ when $\Z_{n}$ acts by rotations of angle $\frac{2\pi}{n}$.
\item {\em Corner reflectors of order $n \geqslant 2$}\index{corner reflector}\index{orbifold!singular locus!corner reflector}: $\R^{2}/D_{n}$ when $D_{n}$ is the dihedral group of order $2n$ generated by reflections in two lines meeting at angle $\frac{\pi}{n}$.
\end{itemize}

In the definition of an orbifold, if we allow $\widetilde{U}_i$ to be an open set in the closed half-space $\R_+^{d}$ of $\R^{d}$, then we obtain the structure of an {\em orbifold with boundary}\index{orbifold!boundary}. To make a somewhat redundant remark, we should not confuse the boundary $\partial\mathcal{O}$ of an orbifold $\mathcal{O}$ with 
the boundary $\partial X_\mathcal{O}$ of the underlying space $X_\mathcal{O}$, when $X_\mathcal{O}$ is a manifold with boundary.

\begin{ex}
A manifold $M$ with boundary $\partial M$ can have an orbifold structure in which $\partial M$ becomes a {\em mirror}\index{singular!locus!mirror}\index{mirror}\index{orbifold!singular locus!mirror}, i.e. a neighborhood of any point $x$ in $\partial M$ is orbifold-diffeomorphic to $\R^d/\mathbb{Z}_2$ such that $\mathbb{Z}_2$ acts by reflection. Notice that the singular locus is then $\partial M$ and the boundary of the orbifold is empty.
\end{ex}

\par{
Given a compact orbifold $\orb$, we can find a cell decomposition of the underlying space $X_{\orb}$ such that the isotropy group of each open cell is constant. Define the orbifold Euler characteristic to be
\[ \chi(\orb) := \sum_{c_{i}} (-1)^{\dim c_{i}} \frac{1}{|\Gamma(c_{i})|}. \] Here, $c_{i}$ ranges over the cells and $|\Gamma(c_{i})|$ is the order of
the isotropy group $\Gamma(c_{i})$ of any point in the relative interior of $c_{i}$.
}
\\
\par{
A {\em covering orbifold}\index{orbifold!covering}\index{covering!orbifold} of an orbifold $\mathcal{O}$ is an orbifold $\mathcal{O}'$ with a continuous surjective map $p \,:\, X_{\mathcal{O}'} \rightarrow X_{\mathcal{O}}$ between the underlying spaces such that each point $x \in X_{\mathcal{O}}$ lies in a coordinate neighborhood $U \approx \widetilde{U}/\Gamma$ and each component $V_i$ of $p^{-1}(U)$ is 
orbifold-diffeomorphic to $\widetilde{U}/\Gamma_i$ with $\Gamma_i$ a subgroup of $\Gamma$. The map $p$ is called a {\em covering map}\index{covering!map}.
}

\begin{ex}
If a group $\Gamma$ acts properly discontinuously on a manifold $M$ and $\Gamma'$ is a subgroup of $\Gamma$, then $M/\Gamma'$ is a covering orbifold of $M/\Gamma$ with covering map $M/\Gamma' \rightarrow M/\Gamma$. In particular, $M$ is a covering orbifold of $M/\Gamma$.
\end{ex}

\par{
Even if it is more delicate than for manifold, we can define the universal covering orbifold of an orbifold $\mathcal{O}$: A {\em universal covering orbifold}\index{universal covering orbifold}\index{orbifold!universal covering} of $\mathcal{O}$ is a covering orbifold $\widetilde{\mathcal{O}}$ with covering map $p : \widetilde{\mathcal{O}} \rightarrow \mathcal{O}$ such that for every covering orbifold $\widetilde{\mathcal{O}}'$ with covering map $p' : \widetilde{\mathcal{O}}' \rightarrow \mathcal{O}$, there is a covering map $q : \widetilde{\mathcal{O}} \rightarrow \widetilde{\mathcal{O}}'$ which satisfies the following commutative diagram:

$$
\begin{tikzpicture}
\matrix(m)[matrix of math nodes,
row sep=1em, column sep=4em,
text height=1.5ex, text depth=0.25ex]
{\widetilde{\mathcal{O}}  & \\
                          & \widetilde{\mathcal{O}}' \\
\mathcal{O}               & \\};
\path[->,font=\small]
(m-2-2) edge node[below]{$p'$}  (m-3-1)
(m-1-1) edge node[left]{$p$}   (m-3-1)
(m-1-1) edge node[above]{$q$} (m-2-2);
\end{tikzpicture}
$$
}

\par{
It is important to remark that every orbifold has a unique universal covering orbifold (up to orbifold-diffeomorphism). The {\em orbifold fundamental group}\index{orbifold!fundamental group} $\pi_{1}^{orb}(\mathcal{O})$ of $\mathcal{O}$ is the group of deck transformations of the universal covering orbifold $\widetilde{\mathcal{O}}$.
}

\begin{ex}
If $\Gamma$ is a cyclic group of rotations acting on the sphere $\S^2$ fixing the north and south poles, then the orbifold $\S^2/\Gamma$
is a sphere with two cone points. Therefore its orbifold fundamental group is $\Gamma$, even though the fundamental group of the sphere, which is the underlying space of $\S^2/\Gamma$, is trivial.
\end{ex}

An orbifold $\mathcal{O}$ is {\em good}\index{orbifold!good}\index{good orbifold} if some covering orbifold of $\mathcal{O}$ is a manifold. In this case, the universal covering orbifold $\widetilde{\mathcal{O}}$ is a simply connected manifold and the group $\pi_{1}^{orb}(\mathcal{O})$ acts properly discontinuously on $\widetilde{\mathcal{O}}$.
In other words, a good orbifold is simply a manifold $M$ with a properly discontinuous group action on $M$. Moreover, we have the following good news:

\begin{theorem}[Chapter 3 of Thurston \cite{Thurston:2002}]
Every $(G,X)$-orbifold is good.
\end{theorem}

\subsection{Geometric structures on orbifolds}
\par{
We will discuss the deformation space\index{deformation!space} of geometric structures on an orbifold $\orb$ as Goldman \cite{Goldmanexp} exposed the theory for manifolds. 
}
\\
\par{
Suppose that $M$ and $N$ are $(G,X)$-orbifolds. A map $f : M \rightarrow N$ is a {\em $(G,X)$-map}\index{$(G,X)$-map} if, for each pair of charts
$\phi : \widetilde{U_i}/\Gamma_{i}  \rightarrow U_i \subset M$
from the $(G,X)$-orbifold structure of $M$ and 
$\psi : \widetilde{V_j}/\Gamma_{j} \rightarrow V_j \subset N$ 
from the $(G,X)$-orbifold structure of $N$, 
the composition $\psi^{-1} \circ f \circ \phi$ restricted to $\phi^{-1}(U_i \cap f^{-1}(V_j))$ lifts to 
the restriction of an element of $G$ on the inverse image in $\widetilde U_i$ of $\phi^{-1}(U_i \cap f^{-1}(V_j))$. 
}
\\
\par{
Recall a {\em $(G,X)$-structure}\index{$(G,X)$-structure}\index{geometric structure} on an orbifold $\orb$ is an orbifold-diffeomorphism $f$ from $\mathcal{O}$ to a $(G,X)$-orbifold $S$. Two $(G,X)$-structures $f: \orb \ra S$ and $f': \orb \ra S'$ on $\orb$ are {\em equivalent}\index{geometric structure!equivalent} if the map $f' \circ f^{-1}: S \ra S'$ is isotopic to a $(G,X)$-map from $S$ to $S'$ (in the category of orbifold).
The set of equivalence classes of $(G,X)$-structures on $\orb$ is denoted by $\Def(\orb)$. There is a topology on $\Def(\orb)$ informally defined by stating that two $(G,X)$-structures $f$ and $f'$ are close if the map $f' \circ f^{-1}: S \ra S'$ is isotopic to a map close to a $(G,X)$-map. Below is a formal definition.
}
\\
\par{
The construction of the developing map and the holonomy representation of manifolds extends to orbifolds without difficulty; see  Goldman \cite{Goldmanexp} for manifolds and Choi \cite{msjbook} for orbifolds. 
For a $(G,X)$-orbifold $\orb$, there exists a pair $(D,\rho)$ of an immersion 
$D : \widetilde{\orb} \rightarrow X$ 
and a homomorphism $\rho : \pi_{1}^{orb}(\orb) \rightarrow G$ such that for each $\gamma \in \pi_{1}^{orb}(\orb)$, the following diagram commutes:
$$
\begin{tikzpicture}
\matrix(m)[matrix of math nodes,
row sep=2em, column sep=3em,
text height=1.5ex, text depth=0.25ex]
{\widetilde{\mathcal{O}}  & X \\
\widetilde{\mathcal{O}}   & X\\};
\path[->,font=\small]
(m-1-1) edge node[left]{$\gamma$}  (m-2-1)
(m-1-1) edge node[above]{$D$} (m-1-2)
(m-2-1) edge node[below]{$D$} (m-2-2)
(m-1-2) edge node[right]{$\rho(\gamma)$} (m-2-2);
\end{tikzpicture}
$$
We call $D$ a {\em developing map}\index{developing map}\index{geometric structure!developing map} and $\rho$ a {\em holonomy representation}\index{holonomy representation}\index{geometric structure!holonomy representation} of $\mathcal{O}$.
Moreover if $(D',\rho')$ is another such pair, then there exists $g \in G$ such that 
\[D'=g \circ D \hbox{ and } \rho' (\gamma)= g \rho(\gamma)g^{-1} \hbox{ for each } \gamma \in \pi_{1}^{orb}(\orb).\] 
In other words, a {\em developing pair}\index{developing pair}\index{geometric structure!developing pair} $(D,\rho)$ is uniquely determined up to the action of $G$: 
\begin{equation}\label{eq:conjugation}
g \cdot (D, \rho(\cdot)) = (g \circ D, g \rho(\cdot) g^{-1}), \textrm{ for each } g \in G.
\end{equation}
}
\\
\par{
Consider the space
\begin{multline*}
\Def'_2(\orb) = \{(D, \rho) \,|\, D: \torb \ra X \hbox{ is an immersion }\\
\textrm{ equivariant with respect to a homomorphism } \rho: \pi_{1}^{orb}(\orb) \ra G \}/\sim.
\end{multline*}
Here $(D, \rho) \sim (D', \rho')$ if $D'= D\circ \tilde \iota$ for the lift $\tilde \iota$ of an isotopy $\iota:\orb \ra \orb$
satisfying $\gamma \circ \tilde \iota = \tilde \iota \circ \gamma$ for every $\gamma \in \pi_{1}^{orb}(\orb)$. We topologize this space naturally using the $C^{r}$-compact-open topology, $r \geq 2$, before taking the quotient, and denote by $\Def_2(\orb)$ the quotient space of $\Def'_2(\orb) $ by the action of $G$ (see Equation (\ref{eq:conjugation})).
}
\\
\par{
We can define a map $\mu:\Def_2(\orb) \ra \Def(\orb)$ from $[ (D, \rho) ]$ to a $(G,X)$-structure on $\orb$ by pulling back the canonical  $(G,X)$-structure on $X$ to $\torb$ by $D$ and taking the orbifold quotient. The inverse map is derived from the construction of the developing pair, hence $\mu$ is a bijection. This gives a topology on $\Def(\orb)$.
}

\subsection{Ehresmann-Thurston principle}
\par{
One of the most important results in this area is the following theorem first stated for closed manifolds. However, it can be easily generalized to closed orbifolds. There exist many proofs of this theorem for manifolds; see Canary-Epstein-Green \cite{CaEp}, Lok \cite{Lok} following John Morgan, Bergeron-Gelander \cite{BergeronGe}, Goldman \cite{Goldmanexp}. For a proof for orbifolds, see Choi \cite{Gorb}, which is a slight modification of the proof for manifolds.
}
\\
\par{
Suppose that $G$ is the real points of a reductive algebraic group defined over $\R$. A representation $\rho : \pi_{1}^{orb}(\orb) \ra G$ is {\em stable}\index{representation!stable}\index{stable representation} when $\rho$ is reductive and the centralizer of $\rho$ is finite.\footnote{A representation $\rho \in \Hom(\pi_{1}^{orb}(\orb), G)$ being stable is equivalent to the fact that the image of $\rho$ is not contained in any parabolic subgroup of $G$ (see Johnson-Millson \cite{JM})} Denote by $\Hom^{st}(\pi_{1}^{orb}(\orb), G)$ the space of stable representations. It is shown in Johnson-Millson \cite{JM} that this is an open subset of $\Hom(\pi_{1}^{orb}(\orb), G)$ and that the action of $G$ on $\Hom^{st}(\pi_{1}^{orb}(\orb), G)$ is proper. Denote by $\Def^{st}_2(\orb)$ the space of $(G,X)$-structures on $\orb$ whose holonomy representation is stable.
}
\begin{theorem}[Ehresmann-Thurston principle]\label{thm:EhTh1}
Let $\orb$ be a closed orbifold. Then the map 
\[\Def'_2(\orb) \to \Hom(\pi_{1}^{orb}(\orb),G) \quad {\rm and} \quad \Def^{st}_2(\orb) \to \Hom^{st}(\pi_{1}^{orb}(\orb),G)/G \] 
induced by $(D, \rho) \ra \rho$ are local homeomorphisms.
\end{theorem}

This principle means that sufficiently nearby $(G,X)$-structures are completely determined by their holomony representations.  

\section{A starting point for convex projective structures}

\subsection{Convexity in the projective sphere or in the projective space}$\,$

\par{
Let $V$ be a real vector space of dimension $d+1$. Consider the action of $\R^*_+$ on $V$ by homothety, and the {\em projective sphere}\index{projective!sphere} 
$$ \S(V) = (V \smallsetminus \{ 0\}) \,/\, \R^*_+ = \{ \textrm{rays of } V\}  .$$
Of course, $\S(V)$ is the 2-fold cover of the real projective space $\Pb(V)$. 
The canonical projection map $V \smallsetminus \{ 0\} \rightarrow \S(V)$ is denoted by $\S$. 
}
\\
\par{A convex cone $\Cc$ is {\em sharp}\index{convex!cone!sharp}\index{sharp convex cone} if $\Cc$ does not contain an affine line.  A subset $C$ of $\S(V)$ is {\em convex}\index{convex!set} (resp. {\em properly convex}\index{properly convex!set}\index{convex!set!properly}) if the subset $\S^{-1}(C) \cup \{ 0 \} $ of $V$ is a convex cone (resp. sharp convex cone). Given a hyperplane $H$ of $\S(V)$, we call the two connected components of $\S(V) \smallsetminus H$ {\em affine charts}\index{affine!chart}. An open set $\O \varsubsetneqq \S(V)$ is convex (resp. properly convex) if and only if there exists an affine chart $\A$ such that $\O \subset \A$ (resp. $\overline{\O} \subset \A$) and $\O$ is convex in the usual sense in $\A$. A properly convex set $\O$ is {\em strictly convex}\index{strictly convex set}\index{convex!set!strictly} if every line segment in $\partial\O$ is a point. All these definitions can be made for subset of $\Pb(V)$. The projective space is more common but the projective sphere allows to get rid of some technical issues. It will be clear from the context whether our convex domain is inside $\S(V)$ or $\Pb(V)$.
}
\\
\par{
The group $\SL^{\pm}(V)$ of linear automorphisms of $V$ of determinant $\pm 1$ is identified to the group of automorphisms of $\S(V)$. A {\em properly convex projective structure}\index{properly convex!real projective structure}\index{geometric structure!real projective structure!properly convex} on an orbifold $\orb$ is a $(\PGL(V),\Pb(V))$-structure (or a $(\SL^{\pm}(V),\S(V))$-structure) whose developing map is a diffeomorphism onto a properly convex subset of $\R\Pd$ (or $\S^d$). 
We refer the reader to Section 1 of Marquis \cite{survey_ludo} to see the correspondence between properly convex $(\PGL(V),\Pb(V))$-structures and properly convex $(\SL^{\pm}(V),\S(V))$-structures (see also p.143 of Thurston \cite{thurston:1997}).
}
\\
\par{
From now on, $\B(\orb)$ will denote the space of properly convex  projective structures on an orbifold $\orb$. This is a subspace of the deformation space of real projective structures on $\orb$.
}

\subsection{Hilbert geometries}\label{hilbert_geo_defi}$\,$

\par{
On every properly convex domain $\O$, there exists a distance $d_{\O}$ on $\O$ defined using the cross ratios: take two points $x \neq y \in \O$ and draw the line between them. This line intersects the boundary $\partial \O$ of $\O$ in two points $p$ and $q$. We assume that $x$ is between $p$ and $y$. If $[p:x:y:q]$ denotes the cross ratio of $p, x, y, q$, then the following formula defines a metric (see Figure \ref{disttt}):
}
$$d_{\O}(x,y) =  \displaystyle \frac{1}{2}\log \Big( [p:x:y:q] \Big), \quad \textrm{ for every }  x, y \in \O.$$

\par{
This metric gives to $\O$ the same topology as the one inherited from $\S(V)$. The metric space $(\O,d_{\O})$ is complete and the closed balls in $\Omega$ are compact. The group $\Aut(\O)$ acts on $\O$ by isometries, and therefore properly.
}
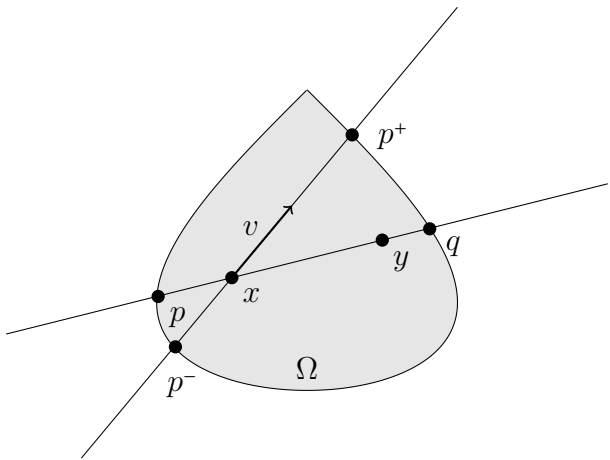
\begin{figure}[h!]
\centering
\begin{tikzpicture}
\filldraw[draw=black,fill=gray!20]
 plot[smooth,samples=200,domain=0:pi] ({4*cos(\x r)*sin(\x r)},{-4*sin(\x r)});
cycle;
\draw (-1,-2.5) node[anchor=north west] {$x$};
\fill [color=black] (-1,-2.5) circle (2.5pt);
\draw (1,-2) node[anchor=north west] {$y$};
\fill [color=black] (1,-2) circle (2.5pt);
\draw [smooth,samples=200,domain=-4:4] plot ({\x},{0.25*\x-2.25});
\draw (-1.97,-2.75) node[anchor=north west] {$p$};
\fill [color=black] (-1.98,-2.75) circle (2.5pt);
\draw (1.7,-1.8) node[anchor=north west] {$q$};
\fill [color=black] (1.63,-1.85) circle (2.5pt);
\draw [smooth,samples=200,domain=-3:2] plot ({\x},{1.2*\x-1.3});
\draw[->,distance=10pt,thick] (-1,-2.5) -- (-0.2,-1.54);
\draw (-1,-1.6) node[anchor=north west] {$v$};
\fill [color=black] (-1.75,-3.42) circle (2.5pt);
\draw (-2,-3.6) node[anchor=north west] {$p^-$};
\fill [color=black] (.6,-.6) circle (2.5pt);
\draw (.8,-.3) node[anchor=north west] {$p^+$};
\draw (0,-3.7) node {$\O$};
\end{tikzpicture}
\caption{The Hilbert metric}
\label{disttt}
\end{figure}

\par{
This metric is called the {\em Hilbert metric}\index{Hilbert!metric} and it can also be defined by a Finsler norm on the tangent space $T_x \O$ at each point $x$ of $\O$: Let $v$ be a vector of $T_x \O$. The quantity  $\left. \frac{d}{dt}\right| _{t=0} d_{\O}(x,x+tv)$ defines a Finsler norm $F_{\O}(x,v)$ on $T_x \O$. Let us choose an affine chart $\A$ containing $\O$ and a Euclidean norm $|\cdot|$ on $\A$.
If $p^+$ (resp. $p^-$) is the intersection point of $\partial \O$ with the half-line determined by $x$ and $v$ (resp. $-v$), and $|\, ab\, |$ is the distance between two points $a,b$ of $\A$ (see Figure \ref{disttt}), then we obtain
}
$$F_{\O}(x,v) = \left. \frac{d}{dt}\right| _{t=0} d_{\O}(x,x+tv) = \frac{|v|}{2}\Bigg(\frac{1}{|xp^-|} + \frac{1}{| xp^+|} \Bigg). $$
\par{The regularity of the Finsler norm is the same as the regularity of the boundary $\partial \O$. The Finsler structure gives rise to an absolutely continuous measure $\mu_{\O}$ with respect to Lebesgue measure, called the {\em Busemann measure}\index{Busemann measure}.
}
\\
\par{
If $\O$ is an ellipsoid, then $(\O,d_{\O})$ is the projective model of the hyperbolic space. More generally, if $\O$ is {\em round}\index{convex!domain!round}, i.e. strictly convex with $\Cc^1$-boundary, then the metric space $(\O,d_{\O})$ exhibits some hyperbolic behaviour
even though $(\O,d_{\O})$ is not Gromov-hyperbolic.\footnote{If $\O$ is strictly convex and $\partial \O$ is of class $\Cc^2$ with a positive Hessian, then $(\O,d_{\O})$ is Gromov-hyperbolic (see Colbois-Verovic \cite{ColVer}), but unfortunately the convex domain we are interested in will be at most round, except the ellipsoid. See also the necessary and sufficient condition of Benoist \cite{benoist_conv_hyp}.}
}
\\
\par{
A properly convex domain $\O$ is a polytope if and only if $(\O,d_{\O})$ is bi-Lipschitz to the Euclidean space \cite{constantin_charc_polytope}. If $\O$ is the projectivization of the space of real positive definite symmetric $m\times m$ matrices, then $\Aut(\O) = \SL_m(\R)$.
}
\\
\par{
Convex projective structures are therefore a special kind of $(G,X)$-structures, whose golden sisters are hyperbolic structures and whose iron cousins are higher-rank structures, i.e. the $(G,G/K)$-structures where $G$ is a semi-simple Lie group without compact factor of (real) rank $\geq 2$, $K$ is the maximal compact subgroup of $G$ and $G/K$ is the symmetric space of $G$.
}
\\
\par{
The above fact motivated an interest in convex projective structures. However, it is probable that this is not the main reason why convex projective structures are interesting. The main justification is the following theorem.
}

\subsection{Koszul-Benoist's theorem}

Recall that the {\em virtual center}\index{virtual center} of a group $\G$ is the subset in $\G$ consisting of the elements whose centralizer is of finite index in $\G$. It is easy to check that the virtual center of a group is a subgroup, and the virtual center of a group is trivial if and only if every subgroup of finite index has a trivial center.

\begin{theorem}[Koszul \cite{kos_open}, Benoist \cite{CD3}]\label{thm-Kos_Ben}
Let $\orb$ be a closed orbifold of dimension $d$ admitting a properly convex real projective structure. 
Suppose that the group $\pi_{1}^{orb}(\orb)$ has a trivial virtual center. 
Then the space $\B(\orb)$ of properly convex projective structures on $\orb$ corresponds to a union of connected components of the character variety $\chi(\pi_{1}^{orb}(\orb),\SL^{\pm}_{d+1}(\R))$.
\end{theorem}

\begin{proof}
Let $\Gamma = \pi_{1}^{orb}(\orb)$ and $G = \SL^{\pm}_{d+1}(\R)$. It was proved by Koszul \cite{kos_open} that the space of holonomy representations of elements of $\B(\orb)$ is open in $\chi(\Gamma,G)$. The closedness follows from Benoist \cite{CD3} together with the following explanation. Assume that $\rho \in \chi(\Gamma,G)$ is an algebraic limit of a sequence of holonomy representations $\rho_{n} \in \chi(\Gamma, G)$ of elements of $\B(\orb)$. This means that $\rho_{n}(\Gamma)$ acts on a properly convex domain $\Omega_{n}$ such that $\Omega_{n} / \rho_{n} (\Gamma) $ is orbifold-diffeomorphic to $\orb$. Then the representation $\rho$ is discrete, faithful and irreducible and $\rho(\Gamma)$ acts on a properly convex domain $\Omega$ by Theorem 1.1 of Benoist \cite{CD3}. We need to show that the quotient orbifold $\mathcal{Q} := \Omega/\rho(\Gamma)$ is orbifold-diffeomorphic to $\orb$:
The openness implies that there is a neighbourhood $U$ of $\rho$ in $\chi(\Gamma, G)$ consisting of holonomy representations of  elements of $\B(\mathcal{Q})$. Since $\rho_{n} \in U$ for sufficiently large $n$, there is a properly convex domain $\Omega'_{n}$ such that $\Omega'_{n}/\rho_{n}(\Gamma)$ is diffeomorphic to $\mathcal{Q}$. Proposition 3 of Vey \cite{Vey} implies that $\Omega'_{n} =\Omega_{n}$ (see also Theorem 3.1 of Choi-Lee \cite{CL15} and Proposition 2.4 of Cooper-Delp \cite{CD}). Since 
\[\mathcal{Q} \approx \Omega'_{n}/\rho_{n}(\Gamma) =\Omega_{n}/\rho_{n}(\Gamma) \approx \orb,\] 
the orbifold $\mathcal{Q}$ is orbifold-diffeomorphic to $\orb$, completing the proof that $\rho$ is the holonomy representation of an element of $\B(\orb)$.
\end{proof}

\begin{rem}
The closedness of Theorem \ref{thm-Kos_Ben} was proved by Choi-Goldman \cite{Choi_gold_close} in dimension $d=2$, and by Kim \cite{ink} in the case that $\Gamma$ is a uniform lattice of $\so{3}(\mathbb{R})$.
\end{rem}

The condition that the virtual center of a group is trivial is transparent as follows:

\begin{propo}[Corollary 2.13 of Benoist \cite{CD3}]\label{prop-Benoist}
Let $\Gamma$ be a discrete subgroup of $\PGL_{d+1}(\R)$.
Suppose that $\Gamma$ acts on a properly convex domain $\Omega$ in $\R \Pb^d$ and $\Omega/\Gamma$ is compact. Then the following are equivalent.
\begin{itemize}
\item The virtual center of $\Gamma$ is trivial.
\item Every subgroup of finite index of $\Gamma$ has a finite center.
\item Every subgroup of finite index of $\Gamma$ is irreducible in $\PGL_{d+1}(\R)$.\footnote{We call $\Gamma$ {\em strongly irreducible}\index{strongly irreducible group} if every finite index subgroup of $\Gamma$ is irreducible.}
\item The Zariski closure of $\Gamma$ is semisimple.
\item The group $\Gamma$ contains no infinite nilpotent normal subgroup.
\item The group $\Gamma$ contains no infinite abelian normal subgroup.
\end{itemize}
\end{propo}

\subsection{Duality between convex real projective orbifolds}

\par{
We start with linear duality. Every sharp convex open cone $C$ of a real vector space $V$ gives rise to a  {\em dual} convex cone\index{dual!convex cone}\index{convex!cone!dual}
$$
C^\star = \{ \varphi \in V^\star \, | \, \varphi_{|\overline{C} \smallsetminus \{ 0\} } > 0 \}.
$$
It can be easily verified that $C^\star$ is also a sharp convex open cone and that
 $C^{\star\star}=C$. Hence, the duality leads to an involution between sharp convex open cones.
}
\\
\par{
Now, consider the “projectivization” of linear duality: If $\O$ is a properly convex domain of $\Pb(V)$ and $C_{\O}$ is the cone of $V$ over $\O$, then {\em the dual $\O^{\star}$ of $\O$}\index{convex!domain!dual}\index{dual!convex domain} is $\Pb(C_{\O}^\star)$. This is a properly convex domain of $\Pb(V^{\star})$. In a more intrinsic way, the dual $\O^\star$ is the set of hyperplanes $H$ of $V$ such that $\Pb(H) \cap \overline{\O} = \varnothing$. Since there is a correspondence between hyperplanes and affine charts of the projective space, the dual $\O^\star$ can be defined as the space of affine charts $\A$ of $\R\Pb^d$ such that $\O$ is a bounded subset of $\A$.
}
\\
\par{
The second interpretation offers us a map $\O^\star \to \O$: namely to $\A \in \O^\star$ we can associate the center of mass of $\O$ in $\A$. The map is well defined since $\O$ is a bounded convex domain of $\A$. In fact, Vinberg showed that this map is an analytic diffeomorphism (see Vinberg \cite{homogeneous_vin_1} or Goldman \cite{wmgnote}), so we call it the {\em Vinberg duality diffeomorphism}\index{Vinberg duality diffeomorphism}.
}
\\
\par{
Finally, we can bring the group in the playground. Recall that the dual representation $\rho^\star:\G \to \mathrm{PGL}(V^{\star})$ of a representation $\rho: \G \to \mathrm{PGL}(V)$ is defined by $\rho^\star(\g) = \,^t\rho(\g^{-1})$, i.e. the dual projective transformation of $\rho(\g)$. All the constructions happen in projective geometry, therefore if a representation $\rho$ preserves a properly convex domain $\O$, then the dual representation $\rho^{\star}$ preserves the dual properly convex domain $\O^\star$. Even more, if we assume the representation $\rho$ to be discrete, then the Vinberg duality diffeomorphism induces a diffeomorphism between the quotient orbifolds $\O/\rho(\G)$ and $\O^\star/\rho^\star(\G)$.
}

\section{The existence of deformations or exotic structures}\label{existence}
\subsection{Bending construction}

\par{
Johnson and Millson \cite{JM} found an important class of deformations of convex real projective structures on an orbifold. The bending construction was introduced by Thurston to deform $(\so{2}(\R),\Hb^2)$-structures on a surface into $(\so{3}(\R),\S^2)$-structures, i.e. conformally flat structures, and therefore in particular to produce quasi-Fuchsian groups. This was extended by Kourouniotis \cite{Kourouniotis} to deform $(\so{d}(\R),\Hb^d)$-structures on a manifold into $(\so{d+1}(\R),\S^d)$-structures.
}
\\
\par{
Johnson and Millson indicated several other deformations, all starting from a hyperbolic structure on an orbifold. However we will focus only on real projective deformations. Just before that we stress that despite the simplicity of the argument, the generalization is not easy. Goldman and Millson \cite{goldman_millson_local_rigidity}, for instance, show that there exists no non-trivial\footnote{A deformation of a representation $\rho : \Gamma \ra G$ is {\em trivial}\index{trivial deformation}\index{deformation!trivial} if it is a conjugation of $\rho$ in $G$.} deformation of a uniform lattice of $\mathrm{SU}_{d,1}$ into $\mathrm{SU}_{d+1,1}$.
}
\\
\par{
Let $\sl{d+1}$ be the Lie algebra of $\SL_{d+1}(\R)$, and let $\orb$ be a closed properly convex projective orbifold. Suppose that $\orb$ contains a two-sided totally geodesic suborbifold $\Sigma$ of codimension $1$. For example, all the hyperbolic manifolds obtained from standard arithmetic lattices of $\so{d}(\R)$, up to a finite cover, admit such a two-sided totally geodesic hypersurface (see Section 7 of Johnson-Millson \cite{JM} for the construction of {\em standard} arithmetic lattices). Let $\Gamma = \pi_{1}^{orb}(\orb)$ and $A=\pi_{1}^{orb}(\Sigma)$. Recall that the Lie group $\SL_{d+1}(\R)$ acts on the Lie algebra $\sl{d+1}$ by the adjoint action.
}

\begin{lemma}[Johnson-Millson \cite{JM}] \label{lem:JM}
Let $\rho \in \Hom(\Gamma,  \SL_{d+1}(\R))$.
Suppose that $\rho(A)$ fixes an element $x_{1}$ in $\sl{d+1}$ and that $x_{1}$ is not invariant under $\Gamma$.
Then $\Hom(\Gamma, \SL_{d+1}(\R))$ contains a non-trivial curve $(\rho_t)_{t \in (- \epsilon, \epsilon)}$, $\epsilon > 0$, with $\rho_0=\rho$, i.e. the curve is transverse to the conjugation action of $\SL_{d+1}(\R)$. 
\end{lemma}

\begin{theorem}[Johnson-Millson \cite{JM}, Koszul \cite{kos_open}, Benoist \cite{CD1,CD3}] \label{thm:JM}
Suppose that a closed hyperbolic orbifold $\orb$ contains $r$ disjoint 
totally geodesic suborbifolds $\Sigma_{1}, \dots, \Sigma_{r}$ of codimension-one. Then the dimension of the space $\B(\orb)$ at the hyperbolic structure is greater than or equal to $r$. Moreover, the bending curves lie entirely in $\B(\orb)$ and all the properly convex structures on $\orb$ are strictly convex.
\end{theorem}

We say that a convex domain $\O$ of the real projective space is {\em divisible}\index{convex!domain!divisible}\index{divisible convex domain} (resp. {\em quasi-divisible}\index{quasi-divisible convex domain}\index{convex!domain!quasi-divisible}) if there exists a discrete subgroup $\G$ of $\Aut(\O)$ such that the action of $\G$ on $\O$ is cocompact (resp. of finite Busemann covolume).

\begin{theorem}[Johnson-Millson \cite{JM}, Koszul \cite{kos_open}, Benoist \cite{CD1}] \label{thm:existence1}
For every integer $d \geqslant 2$, there exists a non-symmetric divisible strictly convex domain of dimension $d$.
\end{theorem}

\begin{rem}
Kac and Vinberg \cite{KacVin} made the first examples of non-symmetric divisible convex domains of dimension $2$ using Coxeter groups (see Section \ref{Coxeter}).
\end{rem}

\begin{rem}
The third author with Ballas \cite{exemple_ludo,Ballasp} extended Theorem \ref{thm:existence1} to non-symmetric {\em quasi-divisible}\index{quasi-divisible convex domain} (not divisible) convex domains.
\end{rem}

\subsection{The nature of the exotics}
\par{
Now we have seen the existence of non-symmetric divisible convex domains in all dimensions $d \geqslant 2$. We first remark that so far all the divisible convex domains built are round if they are not the products of lower dimensional convex domains. So we might want to know if we can go further, specially if we can find indecomposable divisible convex domains that are not strictly convex. The first result in dimension $2$ is negative:\footnote{The word “positive/negative” in this subsection reflects only the feeling of the authors.}
}

\subsubsection{Dimension $2$\,{\rm :} Kuiper-Benz\'ecri’s Theorem}

\begin{theorem}[Kuiper \cite{Ku}, Benz\'ecri \cite{Benthesis}, Marquis \cite{Surf_ludo}] A quasi-divisible convex domain of dimension $2$ is round, except the triangle. 
\end{theorem}

The next result in dimension $3$ is positive:

\subsubsection{Dimension $3$\,{\rm :} Benoist's Theorem}

First, we can classify the possible topology for closed convex projective 3-manifold or 3-orbifold.

\begin{theorem}[Benoist \cite{CD4}]\label{CD4_descrip}
If a closed $3$-orbifold $\mathcal{O}$ admits an indecomposable\footnote{A convex projective orbifold is {\em indecomposable}\index{orbifold!real projective!indecomposable} if its holonomy representation is strongly irreducible.} properly convex projective structure, then $\mathcal{O}$ is topologically the union along the boundaries of finitely many $3$-orbifolds each of which admits a finite-volume hyperbolic structure on its interior.
\end{theorem}

Second, these examples do exist.

\begin{theorem}[Benoist \cite{CD4}, Marquis \cite{ecima_ludo}, Ballas-Danciger-Lee \cite{BDL}]\label{CD4_exist}
There exists an indecomposable divisible properly convex domain $\Omega$ of dimension $3$ which is not strictly convex. Moreover, every line segment in $\partial \Omega$ is contained in the boundary of a properly embedded triangle.\footnote{A simplex $\Delta$ in $\Omega$ is {\em properly embedded}\index{properly embedded} if $\partial \Delta \subset \partial \O$.}
\end{theorem}
\par{
At the time of writing this survey, Theorem \ref{CD4_descrip} is valid only for divisible convex domain. However, Theorem \ref{CD4_exist} is true also for quasi-divisible convex domains which are not divisible (see Marquis \cite{Cox_ludo}).
}
\\
\par{
The presence of properly embedded triangles in the convex domain is related to the existence of incompressible Euclidean suborbifolds on the quotient orbifold. 
Benoist and the third author made examples using Coxeter groups and a work of Vinberg \cite{MR0302779}. We will explore more this technique in Section \ref{Coxeter}. Ballas, Danciger and the second author \cite{BDL} found a sufficient condition under which the double of a cusped hyperbolic three-manifold admits a properly convex projective structure, to produce the examples.
}
\\
\par{
In order to obtain the quasi-divisible convex domains in \cite{Cox_ludo}, the third author essentially keeps the geometry of the cusps. In other words, the holonomy of the cusps preserves an ellipsoid as do the cusps of finite-volume hyperbolic orbifolds. From this perspective, there is a different example whose cusps vary in geometry.
}
\begin{theorem}[Ballas \cite{8knot_sam}, Choi \cite{choi_geo_fini}]
There exists an indecomposable quasi-divisible (necessarily not divisible) properly convex domain $\Omega$ of dimension $3$ which is not strictly convex nor with $\Cc^1$-boundary such that the quotient is homeomorphic to a hyperbolic manifold of finite volume. 
\end{theorem}
\par{
More precisely, the example of Ballas is an explicit convex projective deformation of the hyperbolic structure on the figure-eight knot complement. Note that the first author gave such an example in Chapter 8 of \cite{choi_geo_fini} earlier without verifying its property.  
}
\\
\par{
We do not survey results about cofinite volume action of discrete groups on Hilbert geometries, and we refer the reader to: Marquis \cite{Surf_ludo,Cox_ludo}, Crampon-Marquis \cite{geo_fini}, Cooper-Long-Tillmann \cite{clt_cusps}, Choi \cite{choi_geo_fini}.
}

\subsubsection{Orbifolds of dimension 4 and beyond}

Until now, there are only three sources for non-symmetric divisible convex domains of dimension $d \geqslant 4$: from the “standard” bending of Johnson-Millson \cite{JM}, from the “clever” bending of Kapovich \cite{Kp2}, and using Coxeter groups. The last method was initiated by Benoist \cite{CD4,MR2295544} and extended by the three authors \cite{CLM_ecima}:

\begin{theorem}[Benoist \cite{CD4}, Choi-Lee-Marquis \cite{CLM_ecima}]
For $d=4, \dotsc, 7$, there exists an indecomposable divisible convex domain $\O$ of dimension $d$ which is not strictly convex nor with $\Cc^1$-boundary such that $\O$ contains a properly embedded $(d-1)$-simplex. Moreover, the quotient is homeomorphic to the union along the boundaries of finitely many $d$-orbifolds each of which admits a finite-volume hyperbolic structure on its interior.
\end{theorem}

Other examples were built by Benoist in dimension $4$ and by Kapovich in every dimension:

\begin{theorem}[Benoist \cite{MR2295544}, Kapovich \cite{Kp2}]
For $d \geqslant 4$, there exists a divisible convex domain $\O$ of dimension $d$ such that $(\O,d_{\O})$ is Gromov-hyperbolic but it is not quasi-isometric to a symmetric space. In particular, $\O$ is strictly convex. However it is not quasi-isometric to the hyperbolic space $\mathbb{H}^d$.
\end{theorem}

The three authors recently construct somehow different examples: 

\begin{theorem}[Choi-Lee-Marquis \cite{CLM16}] \label{thm:Dh}
For $d=4, \dotsc, 7$, there exists an indecomposable divisible convex domain $\O$ of dimension $d$ which is not strictly convex nor with $\Cc^1$-boundary such that $\O$ contains a properly embedded $(d-2)$-simplex but does not contain a properly embedded $(d-1)$-simplex.
\end{theorem}

\section{Real projective surfaces}

Another motivation for studying convex projective structures is that these structures are just the right building blocks for all projective structures on closed surfaces. 

\subsection{Affine and projective structures on tori}
\subsubsection{Classification of affine surfaces}

Compact affine surfaces are topologically restrictive:

\begin{theorem}[Benz\'ecri \cite{Benthesis}]
If $S$ is a compact affine surface with empty or geodesic\footnote{An arc of a projective (or affine) surface is {\em geodesic}\index{geodesic!arc} if it has a lift which is developed into a line segment.} boundary, then $\chi(S) = 0$.
\end{theorem}
In the early 1980s, Nagano and Yagi \cite{NY} classified the affine structures on a torus and an annulus with geodesic boundary; see Benoist \cite{BenNil} for a modern viewpoint and Baues \cite{baues} for an extensive survey on this topic.
\\
\par{
Let $Q$ be the closed positive quadrant of $\R^{2}$ and let $U$ be the closed upper-half plane. 
An {\em elementary affine annulus}\index{elementary!affine annulus} is the quotient of 
$Q \smallsetminus \{ 0 \}$ by the group generated by a diagonal linear automorphism with all eigenvalues greater than $1$
or the quotient of $U \smallsetminus \{ 0 \}$ by the group generated by a linear automorphism 
$\left(\begin{smallmatrix}
\lambda & \mu  \\
0 & \lambda \\
\end{smallmatrix}\right)$
with $\lambda > 1$. It is indeed an affine annulus with geodesic boundary. An affine torus $A$ is {\em complex}\index{complex affine torus} if its affine structure comes from a $(\mathbb{C}\smallsetminus \{ 0 \}, \mathbb{C}^{*})$-structure (see page 112 of Thurston \cite{thurston:1997}).
}

\begin{theorem}[Nagano-Yagi \cite{NY}]
If $A$ is a compact affine surface with empty or geodesic boundary,
then one of the following holds\,{\em :}
\begin{itemize}
\item The universal cover of $A$ is either a complete affine space, a half-affine space, a closed parallel strip or a quadrant.
\item The surface $A$ is a complex affine torus.
\item The surface $A$ is decomposed into elementary affine annuli along simple closed geodesics.
\end{itemize}
\end{theorem}

\subsubsection{Classification of projective tori}
Let $\gamma$ be an element of $\SL_{3}( \R)$.
A matrix $\gamma$ is {\em positive hyperbolic}\index{matrix!positive hyperbolic} if $\gamma$ has three distinct positive eigenvalues, that is, $\gamma$ is conjugate to
\[
\left(
\begin{array}{ccc}
 \lambda & 0  & 0  \\
 0 & \mu  & 0  \\
0  & 0  & \nu
\end{array}
\right)\quad  (\lambda\mu\nu=1 \textrm{ and } 0 < \lambda < \mu < \nu ).
\]
A matrix $\gamma$ is {\em planar}\index{matrix!planar} if $\gamma$ is diagonalizable and it has only two distinct eigenvalues. 
A matrix $\gamma$ is {\em quasi-hyperbolic}\index{matrix!quasi-hyperbolic} if $\gamma$ has only two distinct positive eigenvalues and it is not
diagonalizable, that is, $\gamma$ is conjugate to
\begin{equation}\label{matrix:quasi}
\left(
\begin{array}{ccc}
 \lambda & 1  & 0  \\
0  & \lambda   & 0  \\
0  & 0  & \lambda^{-2}
\end{array}
\right)\quad  (0 < \lambda \ne 1).
\end{equation}
A matrix $\gamma$ is a {\em projective translation}\index{matrix!projective translation} (resp. {\em parabolic}\index{matrix!parabolic}) if $\gamma$ is conjugate to
\[
\left(
\begin{array}{ccc}
 1 & 1  & 0  \\
 0 & 1  & 0  \\
0  & 0  & 1
\end{array}
\right) \textrm{ (resp. }
\Bigg(
\begin{array}{ccc}
 1 & 1  & 0  \\
 0 & 1  & 1  \\
0  & 0  & 1
\end{array}
\Bigg)).
\]

\par{
These types of matrices represent all conjugacy classes of non-trivial elements of $\SL_3(\R)$ with positive eigenvalues.
}
\\
\par{
Let $\vartheta$ be a positive hyperbolic element of $\SL_{3}( \R)$ with eigenvalues $\lambda < \mu < \nu$. It is easy to describe the action of $\vartheta$ on the projective plane. This action preserves three lines meeting at the three fixed points. The fixed point $r$ (resp. $s$, $a$) corresponding to the eigenvector for $\lambda$ (resp. $\mu$, $\nu$) is said to be {\em repelling}\index{point!repelling} (resp. {\em saddle}\index{point!saddle}, {\em attracting}\index{point!attracting}).
}
\\
\begin{figure}[h]
\centerline{ \includegraphics[height=9cm]{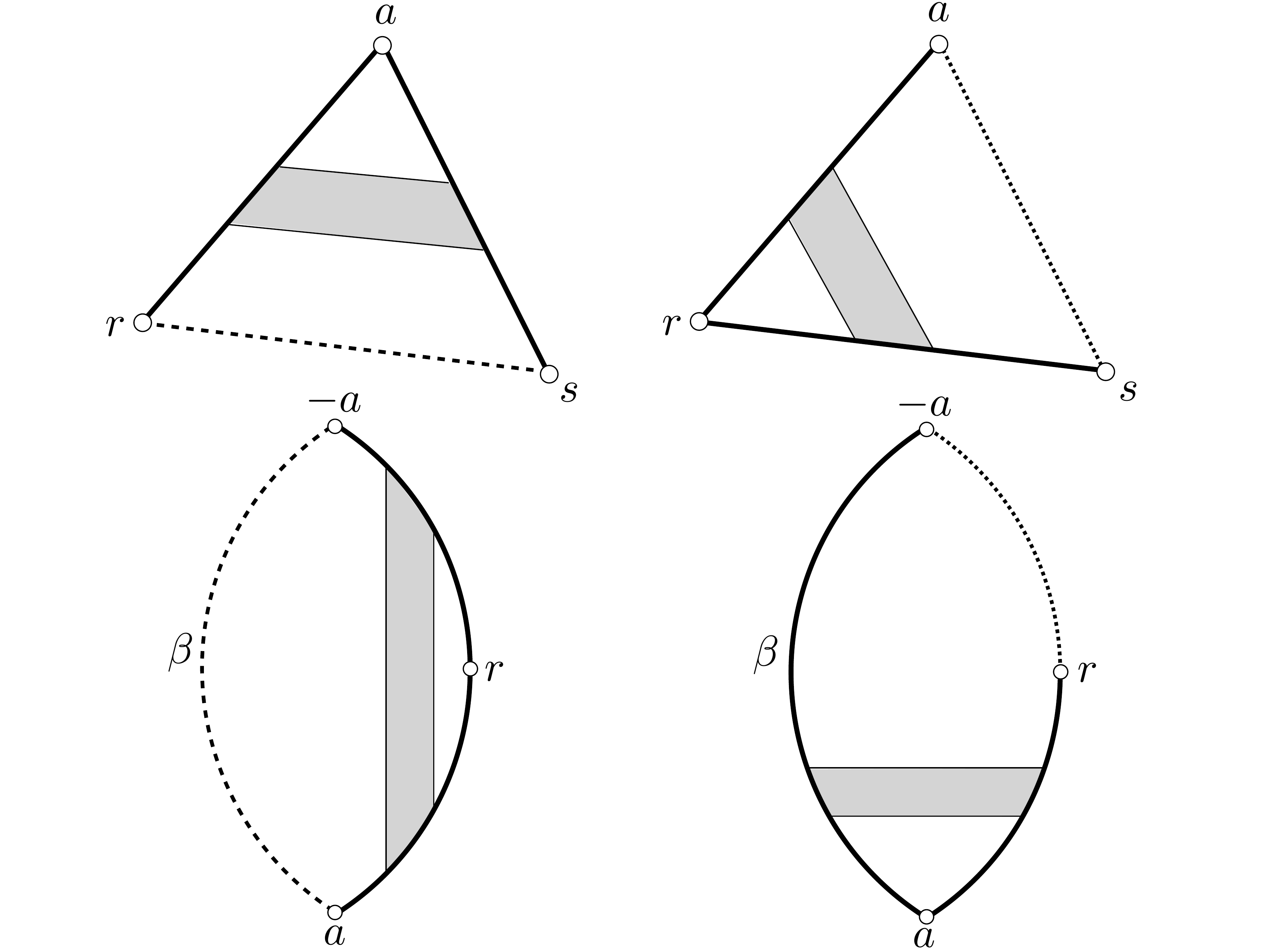}}
\caption{Elementary annuli of type I and type II}
\label{fig:annuli}
\end{figure}
\par{
Let $\triangle$ be a triangle with vertices $r, s, a$.\footnote{To be precise, in $\R\Pb^2$ (resp. $\S^2$) there are four (resp. eight) different choices for $\triangle$ and a unique choice (resp. two choices) for each $r,s,a$. However, once the triangle $\triangle$ in $\S^2$ is chosen, the points $r,s,a$ of $\S^2$ are uniquely determined. Of course, this choice does not affect the later discussion.} An {\em elementary annulus of type I}\index{elementary!annulus of type I} is one of the two\footnote{If $P = \triangle \cup \overline{as}^{\circ} \cup \overline{rs}^{\circ}$, then the space $P / \langle \vartheta \rangle$ is not Hausdorff: Let $x$ be a point of $\overline{as}^{\circ}$, and let $y$ be a point of $\overline{rs}^{\circ}$. For any  neighborhood $\U, \V$ of $x,y$, respectively, in $P$, there exists $N \in \mathbb{N}$ such that $\vartheta^{-N}(\U) \cap \vartheta^N(\V) \neq \varnothing$.} real projective annuli $(\triangle \cup \overline{ar}^{\circ} \cup \overline{as}^{\circ}) /\langle \vartheta \rangle$ and $(\triangle \cup \overline{ar}^{\circ} \cup \overline{rs}^{\circ}) /\langle \vartheta \rangle$ (see Figure \ref{fig:annuli}). These annuli are in fact compact since we can find a compact fundamental domain. We call the image of $ \overline{ar}^{\circ}$ in the annulus a {\em principal}\index{geodesic!principal} closed geodesic, and the image of $\overline{as}^{\circ}$ and the image of $\overline{rs}^{\circ} $ {\em weak}\index{geodesic!weak} closed geodesics.
}
\\
\par{
Now, let $\vartheta$ be a quasi-hyperbolic element conjugate to the matrix (\ref{matrix:quasi}) with $\lambda > 1$. In this case, it is easier to give a description of $\vartheta$ in $\S^2$ than in $\R\Pb^2$, hence we work in $\S^2$. The fixed point $a$ (resp. $r$) corresponding to the eigenvalue $\lambda$ (resp. $\lambda^{-2}$) is attracting (resp. repelling). Let $C$ be the line on which the action of $\vartheta$ is parabolic, and let $\beta$ be the invariant segment in $C$ with endpoints $a$ and $-a$, the antipodal point of $a$, such that, for each point $x \in \beta^{\circ}$, the sequence $\vartheta^n(x)$ converges to $a$. 
}
\\
\par{
Let $L$ be an open lune bounded by $\beta$, $\overline{ar}$, $\overline{-ar}$.
An {\em elementary annulus of type II}\index{elementary!annulus of type II} is one of two real projective annuli $(L \cup \overline{ar}^{\circ} \cup \overline{-ar}^{\circ}) /\langle \vartheta \rangle$ and $(L \cup \beta^{\circ} \cup \overline{ar}^{\circ}) /\langle \vartheta \rangle$ (see Figure \ref{fig:annuli}). These annulli are also compact with geodesic boundary. We call the image of $\overline{ar}^{\circ}$ and of $\overline{-ar}^{\circ}$ in the annulus a {\em principal}\index{geodesic!principal} closed geodesic, and the image of $\beta^{\circ} $ a {\em weak}\index{geodesic!weak} closed geodesic.
}
\\
\par{
By pasting the boundaries of finitely many compact elementary annuli, we obtain an annulus or a torus. The gluing of course requires that boundaries are either both principal or both weak, and their holonomies are conjugate to each other.
}
\begin{theorem}[Goldman \cite{Goldman1977}]
If $T$ is a projective torus or a projective annulus with geodesic boundary,
then $T$ is an affine torus or $T$ can be decomposed into elementary annuli.
\end{theorem}

\subsection{Automorphisms of convex 2-domain}

\par{
A $2$-orbifold $\orb$ is of {\em finite-type}\index{orbifold!finite-type}\index{finite-type!orbifold} if the underlying space of $\orb$ is a surface of finite-type\footnote{A surface of {\em finite-type}\index{finite-type!surface} is a compact surface with a finite number of points removed.} and the singular locus of $\orb$ is a union of finitely many suborbifolds of dimension $0$ or $1$. An element $\g$ of $\pi^{orb}_1(\orb)$ is {\em peripheral}\index{peripheral element} if 
$\g$ is isotopic to an element of $\pi_1^{orb}(\orb \smallsetminus (C \cup \partial \orb))$ for every compact subset $C$ of $\orb \smallsetminus \partial \orb$.
}
\\
\par{
Let $\Cc$ be a properly convex subset of $\SI^{2}$ and let $\G$ be a discrete subgroup of $\mathrm{SL}_3(\mathbb{R})$ acting properly discontinuously on $\Cc$. A closed geodesic $g$ in $S=\Cc/\G$ is {\em principal}\index{geodesic!principal} when the holonomy $\gamma$ of $g$ is positive hyperbolic or quasi-hyperbolic, and the lift of $g$ is the geodesic segment in $\Cc$ connecting the attracting and repelling fixed points of $\g$. In addition, if $\gamma$ is positive hyperbolic, then $g$ is said to be {\em $h$-principal}\index{geodesic!$h$-principal}.
}
\\
\par{
The following theorem generalizes well-known results of hyperbolic structures on surfaces,
and is essential to understand the convex real projective $2$-orbifolds. The nonorientable orbifold version exists but is a bit more complicated to state (see Choi-Goldman \cite{cgorb}).
}

\begin{theorem}[Kuiper \cite{Ku}, Choi \cite{cdcr1}, Marquis \cite{Surf_ludo}]
Let $S = \Cc/\G$ be an orientable properly convex real projective $2$-orbifold of finite type with empty or geodesic boundary, 
for a properly convex subset $\Cc \subset \SI^{2}$. We denote by $\O$ the interior of $\Cc$. Suppose that $\O$ is not a triangle. 
\begin{enumerate}
\item An element has finite order if and only if it fixes a point in $\O$.
\item Each infinite-order element $\gamma$ of $\Gamma$ is positive hyperbolic, quasi-hyperbolic or parabolic.

\item If an infinite-order element $\gamma$ is nonperipheral, then $\gamma$ is positive hyperbolic and a unique closed geodesic $g$ realizes $\gamma$. Moreover, $g$ is principal and any lift of $g$ is contained in $\O$. If $\gamma$ is represented by a {\em simple} closed curve in $S$, then $g$ is also {\em simple}.

\item A peripheral positive hyperbolic element is realized by a unique principal geodesic $g$, and either all the lifts of $g$ are contained in $\O$ or all the lifts of $g$ are in $\Cc \cap \partial \O$.

\item A quasi-hyperbolic element $\g$ is peripheral and is realized by a unique geodesic $g$ in $\Cc$. Moreover, $g$ is principal and any lift of $g$ is in $\Cc \cap \partial \O$.

\item A parabolic element $\g$ is peripheral and is realized by the projection of $\E \smallsetminus \{p \}$ where $\E$ is a $\g$-invariant ellipse whose interior is inside $\O$ and $p \in \partial \O$ is the unique fixed point of $\g$.
\end{enumerate}
\end{theorem}
Note that in the fourth item, the closed geodesics homotopic to $g$ may not be unique.

\subsection{Convex projective structures on surfaces}\label{subsection:convexsurfaces}
\par{
Let $S$ be a compact surface with or without boundary. When $S$ has boundary, in general, the holonomy of a convex projective structure on $S$ does not determine the structure. More precisely, there exists a convex projective structure whose holonomy preserves more than one convex domain. Therefore we need to make some assumptions on the convex projective structure in order to avoid this problem hence we consider the subspace, denoted by $\C^{pgb}(S)$, in $\C(S)$ of convex projective structures on $S$ for which each boundary component is a principal geodesic. 
}

\begin{theorem}[Goldman \cite{Gconv}] \label{thm:CS}
If $S$ is a closed surface of genus $> 1$, then the space $\C(S)$ is homeomorphic to an open cell of dimension $-8\chi(S)$.
\end{theorem}
\par{
The following two propositions illustrate the proof of Theorem \ref{thm:CS}. Recall that for a Lie group $G$, a fiber bundle $\pi : P \rightarrow X$ is {\em $G$-principal}\index{principal fiber bundle} if there exists an action of $G$ on $P$ such that $G$ preserves the fibers and acts simply transitively on them.
}
\\
\par{
Let $c$ be a non-peripheral simple closed curve in $S$. The complement of $c$ in $S$ can be connected, or a disjoint union of $S_{1}$ and $S_{2}$. In the first case, the completion $S'$ of $S \smallsetminus c$ has two boundary components $c'_1$ and $c'_2$ corresponding to $c$. Denote by $\C_c^{pgb}(S')$ the subspace of structures in $\C^{pgb}(S')$ satisfying that the holonomies of $c'_1$ and $c'_2$ are both positive hyperbolic and conjugate to each other. In the second case, the completion of each $S_{i}$, $i=1,2$, has a boundary component corresponding to $c$. Denote by $\C^{pgb}_c(S'_{1}) \boxtimes_c \C^{pgb}_c(S'_{2})$ the subspace of structures $(P_1,P_2)$ in $\C^{pgb}(S'_{1}) \times \C^{pgb}(S'_{2})$ 
whose holonomies corresponding to $c$ are positive hyperbolic and conjugate to each other. 
}

\begin{propo}[Goldman \cite{Gconv}, Marquis \cite{marquis_moduli_surf}]\label{prop:bending}
Let $S$ be a compact surface with or without boundary such that $\chi(S)<0$, and let $c$ be a non-peripheral simple closed curve in $S$. If $S \smallsetminus c$ is connected {\em (}resp. a disjoint union of $S_{1}$ and $S_{2}${\em )}, then the forgetful map $\C^{pgb}(S) \to \C^{pgb}_c(S')$ {\em (}resp. $\C^{pgb}(S) \to \C^{pgb}_c(S'_{1}) \boxtimes_c \C^{pgb}_c(S'_{2})${\em )} is an $\R^{2}$-principal fiber bundle.
\end{propo}

\par{
Proposition \ref{prop:bending} should be compared with Lemma \ref{lem:JM}. Since the holonomy $\gamma$ of a non-peripheral simple closed curve $c$ is positive hyperbolic, the centralizer of $\gamma$ is $2$-dimensional. The first gluing parameter is the twist parameter like in hyperbolic geometry, and the second gluing parameter is the bending parameter we obtain in view of Lemma \ref{lem:JM}.
}
\\
\par{
Next, we need to understand the convex projective structures on a pair of pants. 
Assume that $\gamma$ is an element of $\mathrm{SL}_3(\mathbb{R})$ with positive eigenvalues. We denote by $\lambda({\gamma})$ the smallest eigenvalue of $\gamma$ and by $\tau({\gamma})$ the sum of the two other eigenvalues, and we call the pair $(\lambda(\gamma), \tau(\gamma))$ the {\em invariant}\index{invariant} of $\gamma$. The map $[\gamma] \mapsto (\lambda(\gamma), \tau(\gamma))$ is a homeomorphism between the space of conjugacy classes of positive hyperbolic or quasi-hyperbolic elements $\gamma$ of $\mathrm{SL}_3(\mathbb{R})$ and the space
\[\mathcal{R}:= \{(\lambda, \tau) \in \R^{2} \,|\,\,  0 < \lambda < 1,\,\,\, \frac{2}{\sqrt{\lambda}} \leq \tau \leq \lambda + \frac{1}{\lambda^{2}}\} \]
(see Figure \ref{figure:pants}). Note that $(1,2) \notin \mathcal{R}$ corresponds to the conjugacy class of parabolic elements.  
}
\\
\begin{figure}[h]
\centerline{ \includegraphics[height=6cm]{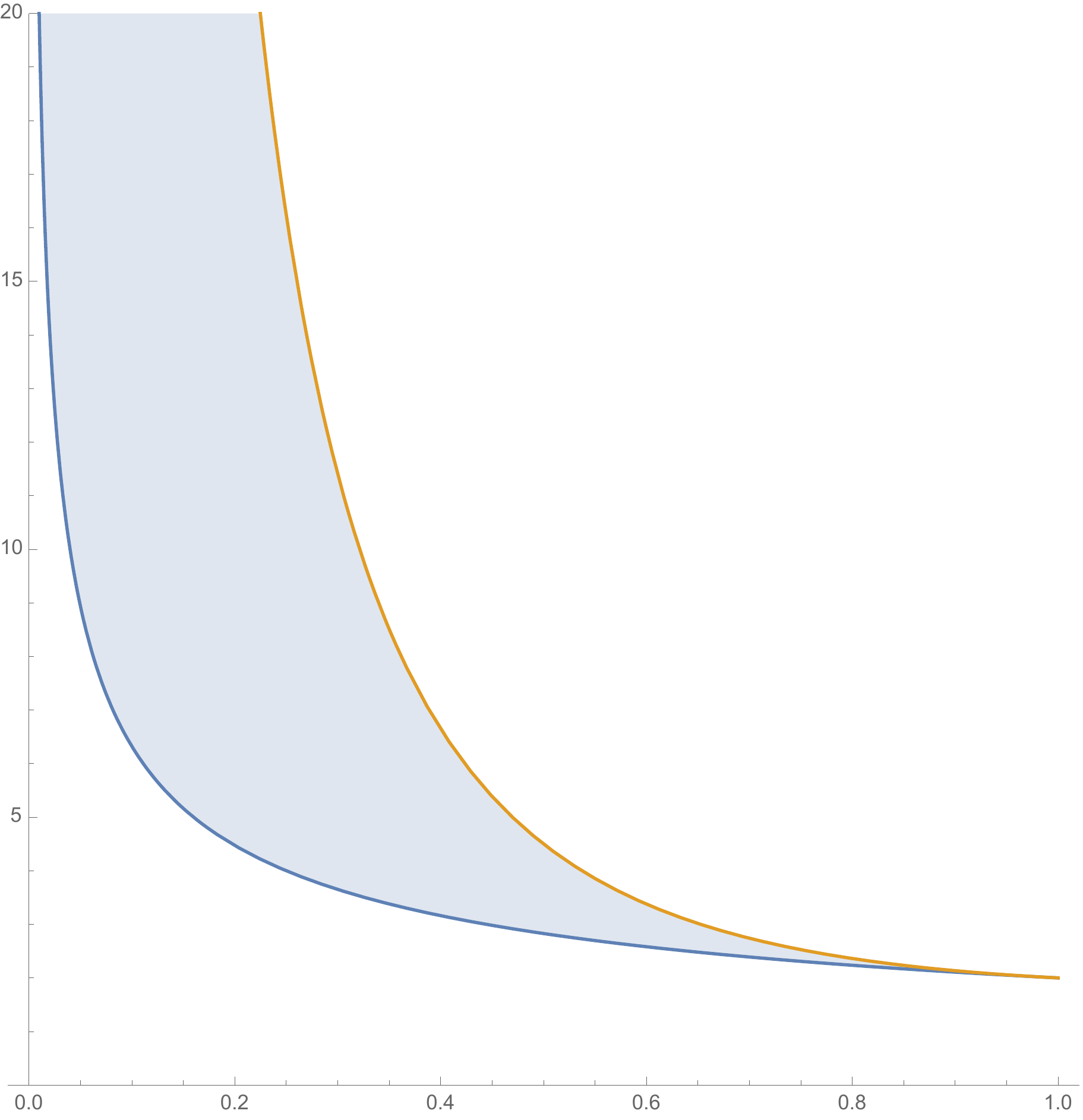}}
 \caption{ The region $\mathcal{R}$ between the graphs $y = \frac{2}{\sqrt{x}}$ and $y= x+ \frac{1}{x^2}$ 
 }\label{figure:pants}
 \end{figure}
\par{
Let $S$ be a compact surface with $n$ (oriented) boundary components $c_{i}$, $i=1, \dotsc, n$. Let us define a map
\[R_{\partial S}: \C^{pgb}(S) \ra \mathcal{R}^{n} \]
from each structure to the $n$-tuples of invariants $(\lambda(\gamma_i), \tau(\gamma_i))$ of the holonomy $\gamma_i$ of $c_{i}$.
}
\begin{propo}[Goldman \cite{Gconv}, Marquis \cite{marquis_moduli_surf}] \label{thm:pp}
If $P$ is a pair of pants, then the map $R_{\partial P}: \C^{pgb}(P) \ra \mathcal{R}^{3} $ is an $\R^2$-principal fiber bundle, and the interior of $\C^{pgb}(P)$ is exactly the space of convex projective structures with $h$-principal geodesic boundary. In particular it is an open cell of dimension $8$.
\end{propo}

\par{
Now consider a surface $S$ of finite type with ends. A properly convex projective structure on $S$ is of {\em relatively finite volume}\index{relatively finite volume} if for each end of $S$, there exists an end neighborhood $\mathcal{V}$ such that $\mu_{S}(\mathcal{V}) < \infty$.
Let $\B^{pgb}_{gf}(S)$ denote the subspace in $\C^{pgb}(S)$ of convex projective structures with principal geodesic boundary and of relatively finite volume. The following theorem generalizes to such structures:
}

\begin{theorem}[Fock-Goncharov \cite{baby_fock}, Theorem 3.7 of Marquis \cite{marquis_moduli_surf}]
Let $S$ be a surface of finite-type. If $b$ is the number of boundary components of $S$ and $p$ is the number of ends of $S$, then the space $\B^{pgb}_{gf}(S)$ is a manifold with corner homemorphic to $\R^{16g-16+6p+7b} \times [0,1]^b$, and the interior of $\B^{pgb}_{gf}(S)$ is exactly the space of structures with $h$-principal geodesic boundary and of relatively finite volume.
\end{theorem}

\subsection{The convex decomposition of projective surfaces of genus $>1$}

Now let us consider compact real projective surfaces with geodesic boundary.

\begin{theorem}[Choi \cite{cdcr1,cdcr2}] \label{thm:convdec}
Let $S$ be a compact projective surface with geodesic boundary. 
If $\chi(S) < 0$, then $S$ can be decomposed along simple closed geodesics into convex projective subsurfaces with principal geodesic boundary and elementary annuli.
\end{theorem}
\par{
The annuli with principal geodesic boundary whose holonomies are positive hyperbolic (resp. quasi-hyperbolic) were classified by Goldman \cite{Goldman1977} (resp. Choi \cite{cdcr2}).
}
\\
\par{
If $A$ is an annulus with quasi-hyperbolic principal geodesic boundary, then by Proposition 5 of Choi \cite{choi_III} only one boundary can be identified with a boundary of a convex real projective surface with principal geodesic boundary.
Hence, if a compact real projective surface $S$ has quasi-hyperbolic holonomy for a closed curve, then $S$ must have a boundary.
}
\begin{rem}
Given a convex projective surface $S = \Omega/\Gamma$ with a $h$-principal boundary and an elementary annulus $A$, we can obtain a new projective surface $S'$ by identifying the respective boundary components 
of $S$ and $A$. The surface $S'$ is still convex since the union of $\Omega$ and the triangles given by the universal cover of $A$ is convex.
\end{rem}

\subsection{Projective structures on a closed surface of genus $> 1$}

\par{
Let $S$ be a closed surface of genus greater than $1$ and let $\Pb(S)$ denote the space of real projective structures on $S$. For each connected component of $\Pb(S)$, any two elements share the same decomposition up to isotopy, given by Theorem \ref{thm:convdec}. Let $\mathcal{S}(S)$ denote the collection of mutually disjoint isotopy classes of non-trivial simple closed curves, and let $F_2^{+,even}$ be the set of elements of the free semigroup on two generators whose word-lengths are even.
}
\\
\par{
In \cite{Goldman1977} Goldman constructs a map \[\Pb(S)  \ra  \mathcal{S}(S) \times F_2^{+,even}\] that describes the gluing patterns of elementary annuli of type I (with $h$-principal geodesic boundary) in a projective structure $P$ on $S$. Finally, by removing all the annuli from $P$ and reattaching it, we obtain a convex projective structure on $S$.
}

\begin{theorem}[Choi \cite{cdcr1,cdcr2}, Goldman \cite{Gconv}] \label{thm-CG}
Let $S$ be a closed surface with $\chi(S) < 0$. Then each fiber of the map $\Pb(S)  \ra  \mathcal{S}(S) \times F_2^{+,even}$ can be identified with $\C(S)$. In particular, $\Pb(S)$ is homeomorphic to a disjoint union of countably many open cells of dimension $16(g-1)$.
\end{theorem}

\subsection{Convex projective orbifolds of negative Euler characteristic}

\par{
Every compact $2$-orbifold $\Sigma$ is obtained from a surface with corners by making some arcs mirrors and putting cone-points and corner-reflectors in a locally finite manner. An endpoint of a mirror arc can be either in a boundary component of $\Sigma$ or a  corner-reflector in $\Sigma$ that should be an endpoint of another mirror arc. Moreover, the smooth topology of a $2$-orbifold is determined by the underlying topology of the surface with corners, the number of cone-points of order $q \in \mathbb{N} \smallsetminus \{0,1\}$, the number of corner-reflectors of order $r \in \mathbb{N} \smallsetminus \{0,1\}$, and the boundary patterns of the mirror arcs.
}
\\
\par{
A {\em full} $1$-orbifold\index{full $1$-orbifold} is a segment with two mirror endpoints. Let $\Sigma$ be a compact $2$-orbifold with $m$ cone points of order $q_{i}$, $i=1, \dotsc, m$, and $n$ corner-reflectors of order $r_{j}$, $j=1, \dotsc, n$, and $n_{\Sigma}$ boundary full $1$-orbifolds. The orbifold Euler characteristic of $\Sigma$ is
\[\chi(\Sigma) = \chi(X_{\Sigma})  - \sum_{i=1}^{m} \left(1- \frac{1}{q_{i}}\right) - \frac{1}{2} \sum_{j=1}^{n} \left( 1- \frac{1}{r_{j}} \right)
-\frac{1}{2}n_{\Sigma}.\]
This is called the generalized Riemman-Hurwitz formula (see Section \ref{subsection:kit} for the definition of the orbifold Euler characteristic). 
}

\begin{theorem}[Thurston \cite{Thurston:2002}] \label{thm:hypo}
Let $\Sigma$ be a compact $2$-orbifold of negative orbifold Euler characteristic with the underlying space $X_{\Sigma}$.
Then the deformation space ${\mathcal{T}}(\Sigma)$ of hyperbolic structures on $\Sigma$
is a cell of dimension $-\chi(X_{\Sigma}) + 2m + n + 2n_{\Sigma}$ where $m$ is the number of
cone-points, $n$ is the number of corner-reflectors and $n_{\Sigma}$ is the number of boundary full $1$-orbifolds.
\end{theorem}

\par{
In order to understand the deformation spaces of convex projective structures on surfaces, we saw that it is important to study the convex projective structures on a pair of pants, which is the most  “elementary” surface. Similarly, in the case of 
$2$-orbifolds, we should firstly understand the elementary $2$-orbifolds.
}
\\
\par{
Let us discuss the process of splitting and sewing of $2$-orbifolds. Note that orbifolds always have a path-metric. For example, we can define a notion of Riemannian metric on an orbifold $\Sigma$, i.e. for all coordinate neighborhoods $U_i \approx \widetilde{U}_i/\Gamma_i$, there exist $\Gamma_i$-invariant Riemannian metrics on $\widetilde{U}_i$ which are compatible each other (see Choi \cite{msjbook}). Let $c$ be a simple closed curve or a full $1$-orbifold in the interior\footnote{The {\em interior}\index{orbifold!interior} of an orbifold $\Sigma$ is $\Sigma \smallsetminus \partial \Sigma$.} of a $2$-orbifold $\Sigma$ and let $\hat \Sigma$ be the completion of $\Sigma \smallsetminus c$ with respect to the path-metric induced from one on $\Sigma$. We say that $\hat \Sigma$ is obtained from the {\em splitting}\index{orbifold!splitting}\index{splitting} of $\Sigma$ along $c$. Conversely, if $\hat c$ is the union of two boundary components of $\hat \Sigma$ corresponding to $c$, then $\Sigma$ is obtained from {\em sewing}\index{orbifold!sewing}\index{sewing} $\hat \Sigma$ along $\hat c$.
}
\\
\par{
 An {\em elementary $2$-orbifold}\index{elementary!$2$-orbifold} is a compact $2$-orbifold of negative orbifold Euler characteristic which we cannot split further along simple closed curves or full $1$-orbifolds into suborbifolds. We assume in this subsection that our orbifolds are of negative orbifold Euler characteristic.
}
\\
\par{
The following is the classification of elementary $2$-orbifolds (see Figure \ref{fig:eleorb}). Arcs with/without dotted arcs next to them indicate boundary/mirror components, respectively, and black points indicate singular points. We can obtain the orbifolds (P$j$), $j =2,3,4$, from changing the boundary components of (P1) to cusps and then to elliptic points, considering them as hyperbolic surfaces with singularities. For each $j=1, \dotsc, 4$, the orbifold (D$j$) (resp. (A$j$)) is the quotient orbifold of (P$j$) by an order-two involution preserving (resp. switching a pair of) boundary components or cone-points.
Note that the underlying space of (P1) is a pair of pants, the ones of (P2), (A1) and (A2) are closed annuli, the ones of (P3), (A3), (A4) and (D1)-(D4) are closed disks, and the one of (P4) is a sphere.
}

\begin{figure}[h]

\centerline{ \includegraphics[height=9cm]{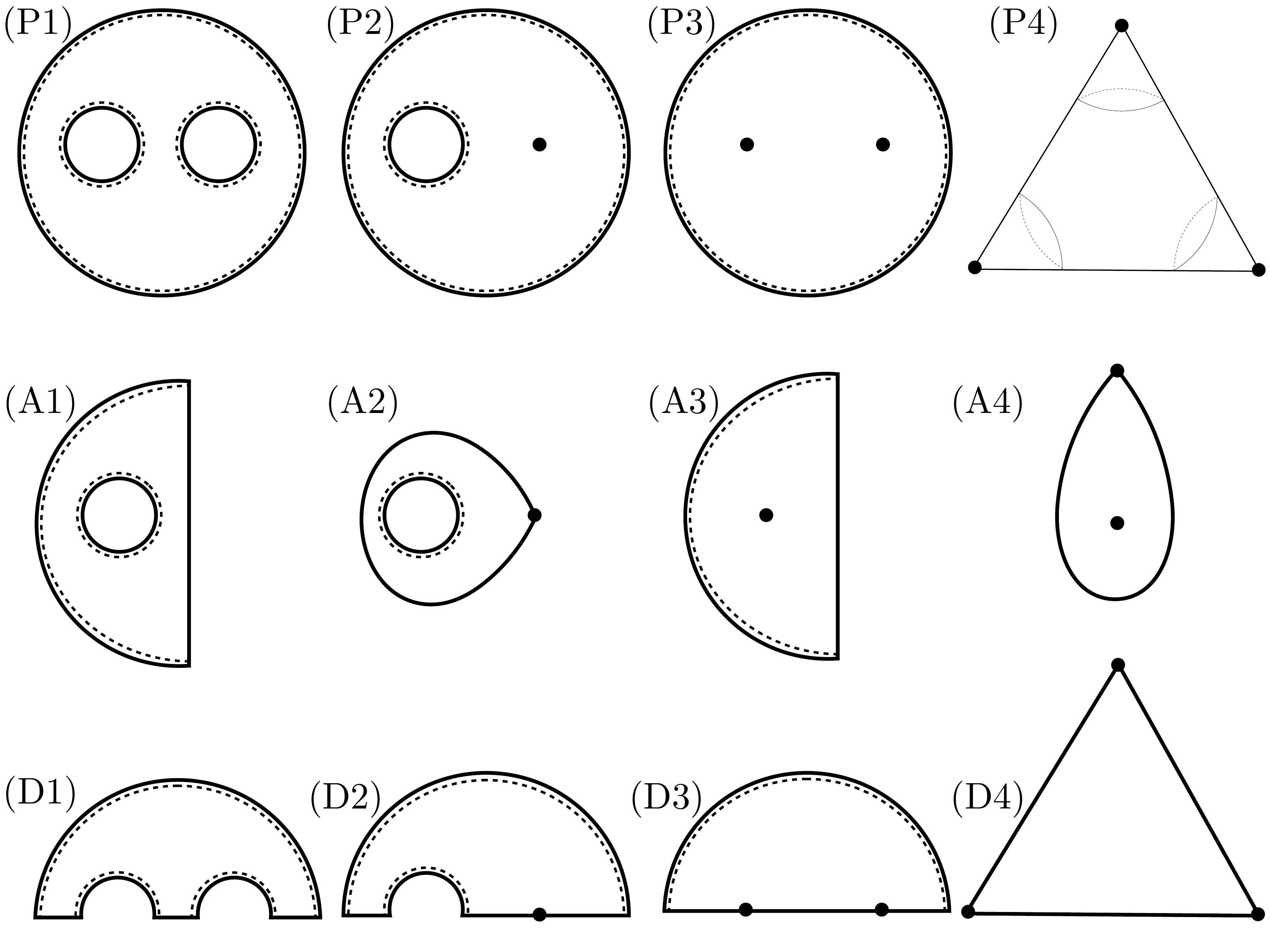}}

\caption{Elementary orbifolds}
\label{fig:eleorb}

\end{figure}

  \begin{enumerate}
\item[(P1)] A pair of pants ($\chi =- 1$).
\item[(P2)] An annulus with a cone-point of order $p$ ($\chi =  \frac{1}{p} -1$).
\item[(P3)] A disk with two cone-points of orders $p, q$ ($\chi = \frac{1}{p} + \frac{1}{q} -1$).
\item[(P4)] A sphere with three cone-points of order
$p, q, r$ ($\chi =  \frac{1}{p} + \frac{1}{q} + \frac{1}{r} -1$).
\end{enumerate}

 \begin{enumerate} 
\item[(A1)] An annulus with a boundary circle, a boundary arc and a mirror arc 
($\chi = -\frac{1}{2}$).

\item[(A2)] An annulus with a boundary circle and 
a corner-reflector of order $p$ ($\chi = \frac{1}{2p} -\frac{1}{2}$).

\item[(A3)] A disk with a boundary arc, a mirror arc
and a cone-point of order $p$ ($\chi = \frac{1}{p} - \frac{1}{2}$).

\item[(A4)] A disk with a corner-reflector of order $p$
and a cone-point of order $q$ ($\chi=\frac{1}{2p} + \frac{1}{q}-\frac{1}{2}$).

\end{enumerate}

  \begin{enumerate}  

\item[(D1)] A disk with three mirror arcs and three boundary arcs ($\chi= -\frac{1}{2}$). 

\item[(D2)] A disk with a corner-reflector of order $p$ at which two mirror arcs meet,  one more mirror arc and two boundary arcs ($\chi = \frac{1}{2p} - \frac{1}{2}$).

\item[(D3)] A disk with two corner-reflectors of order $p$, $q$, and a boundary arc ($\chi =   \frac{1}{2p} + \frac{1}{2q} - \frac{1}{2} $). 
\item[(D4)]  A disk with three corner-reflectors of order $p, q, r$ and three mirror arcs ($\chi =  \frac{1}{2p} + \frac{1}{2q} + \frac{1}{2r} -\frac{1}{2}$).
\end{enumerate}

\par{
Let $\Sigma$ be a properly convex real projective $2$-orbifold. A geodesic full $1$-orbifold $c$ in $\Sigma$ is {\em $h$-principal}\index{full $1$-orbifold!$h$-principal} if there is a double cover $\widetilde{\Sigma}$ of $\Sigma$ such that the double cover of $c$ in $\widetilde{\Sigma}$ is a simple closed geodesic and is $h$-principal. 
}
\\
\par{
Assume that every boundary component of $\Sigma$ is $h$-principal. Let $c$ be an oriented boundary component of $\Sigma$. We know from Section \ref{subsection:convexsurfaces} that if $c$ is homeomorphic to a circle, then the space $\textrm{Inv}(c)$ of projective invariants of $c$ is homeomorphic to $\mathcal{R}^{\circ}$. However, if $c$ is a full $1$-orbifold, then $\textrm{Inv}(c) = \R^*$ because in that case the space $\C(c)$ of convex projective structures on the Coxeter $1$-orbifold $c$ is parametrised by the Hilbert length of $c$ in the universal cover of $c$.
}
\\
\par{
Denoting by $B(\Sigma)$ the set of boundary components of $\Sigma$, we define
\[\textrm{Inv}(\partial \Sigma) :=\prod_{c \in B(\Sigma)} \textrm{Inv}(c)  \quad  \textrm{and} \quad \textrm{Inv}(\emp) = \{\ast\} \textrm{ is a singleton}.\]
}
\begin{propo}[Choi-Goldman \cite{cgorb}] \label{prop:CG}
Let $S$ be an elementary orbifold in Figure \ref{fig:eleorb}. The map
$\C^{pgb}(S)^{\circ} \ra \mathrm{Inv}(\partial S)$ is a fibration of an $n$-dimensional open cell over the $k$-dimensional open cell with $l$-dimensional open cell fiber.
We list $(n, k, l)$ below.
\begin{description}
\item[(P1)] $(8, 6, 2)$ {\rm (}Goldman \cite{Gconv}{\rm )}.
\item[(P2)] $(6, 4, 2)$ if there is no cone-point of order $2$. Otherwise $(4, 4, 0)$.
\item[(P3)] $(4, 2, 2)$ if there is no cone-point of order $2$. Otherwise $(2, 2, 0)$.
\item[(P4)] $(2, 0, 2)$ if there is no cone-point of order $2$. Otherwise $(0, 0, 0)$.
\item[(A1)] $(4, 3, 1)$.
\item[(A2)] $(3, 2, 1)$ if there is no corner-reflector of order $2$. Otherwise $(2, 2, 0)$.
\item[(A3)] $(2, 1, 1)$.
\item[(A4)]  $(1, 0, 1)$ if there is no corner-reflector of order $2$. Otherwise $(0, 0, 0)$.
\item[(D1)] $(4, 3, 1)$.
\item[(D2)] $(3, 2, 1)$ if there is no corner-reflector of order $2$. Otherwise $(2, 2, 0)$.
\item[(D3)] $(2, 1, 1)$ if there is no corner-reflector of order $2$. Otherwise $(1, 1, 0)$.
\item[(D4)] $(1, 0, 1)$  if there is no corner-reflector of order $2$. Otherwise $(0, 0, 0)$.
\end{description}
\end{propo}

Finally, we can describe the deformation space of convex projective structures on closed $2$-orbifolds.

\begin{theorem}[Choi-Goldman \cite{cgorb}] \label{thm:CG}
Let $\Sigma$ be a closed $2$-orbifold with $\chi(\Sigma) < 0$.
Then the space $\C(\Sigma)$ of convex projective structures on $\Sigma$ is
homeomorphic to a cell of dimension
\[ -8\chi(X_{\Sigma}) + (6 m - 2 m_2) + (3n - n_2) \]
where $X_{\Sigma}$ is the underlying space of $\Sigma$, $m$ is the number of cone-points,
$m_2$ is the number of cone-points of order $2$, $n$ is the number of corner-reflectors,
and $n_2$ is the number of corner-reflectors of order $2$.
\end{theorem}

\section{Convex projective Coxeter orbifolds}\label{Coxeter}
\subsection{Definitions}
\subsubsection{Coxeter group}

\par{
Let $S$ be a finite set and denote by $|S|$ the cardinality of $S$. A {\em Coxeter matrix}\index{Coxeter!matrix} on $S$ is an $|S| \times |S|$ symmetric matrix $M=(M_{st})_{s,t \in S}$ with diagonal entries $M_{ss}=1$ and other entries $M_{st} \in \{2, 3, \dotsc, \infty\}$. The pair $(S,M)$ is called a {\em Coxeter system}\index{Coxeter!system}.
}
\\
\par{
To a Coxeter system $(S,M)$ we can associate a {\em Coxeter group}\index{Coxeter!group} $W=W_S$: it is the group generated by $S$ with the relations $(st)^{M_{st}}=1$ for all $(s,t) \in S \times S$ such that $M_{st} \neq \infty$.
If $(S,M)$ is a Coxeter system, then for each subset $T$ of $S$ we can define the Coxeter sub-system $(T,M_T = (M_{st})_{s,t \in T} )$. The Coxeter group $W_T$ can be thought of as a subgroup of $W_S$ since the canonical map from $W_T$ to $W_S$ is an embedding. We stress that the last sentence is not an obvious statement and it is, in fact, a corollary of Theorem \ref{t:tits-vin}.
}
\\
\par{
The {\em Coxeter graph}\index{Coxeter!graph} of a Coxeter system $(S,M)$ is the \textit{labelled} graph where the set of vertices is $S$, two vertices $s$ and $t$ are connected by an edge $\overline{st}$ if and only if $M_{st} \neq 2$, and the label of the edge $\overline{st}$ is $M_{st}$. A Coxeter system $(S,M)$ is {\em irreducible}\index{Coxeter!system!irreducible} when its Coxeter graph is connected. It is a little abusive but we also say that the Coxeter group $W$ is irreducible.
}

\subsubsection{Coxeter orbifolds}

\par{
We are interested in $d$-dimensional Coxeter orbifolds whose underlying space is homeomorphic to a $d$-dimensional polytope\footnote{We implicitly assume that all the polytopes and polygons are convex.} $P$ minus some faces, and whose singular locus is the boundary of $P$ made up of mirrors. For the sake of clarity, facets\index{facet}\index{polytope!facet} are faces of codimension $1$, ridges\index{ridge}\index{polytope!ridge} are faces of codimension $2$ and proper faces\index{proper face}\index{polytope!proper face} are faces different from $P$ and $\varnothing$. Choose a polytope $P$ and a Coxeter matrix $M$ on the set $S$ of facets of $P$ such that if two facets $s$ and $t$ are not adjacent,\footnote{Two facets $s$ and $t$ are {\em adjacent}\index{facet!adjacent}\index{adjacent facet} if $s \cap t$ is a ridge of $P$.} then $M_{st} = \infty$. When two facets $s,t$ are adjacent, the ridge $s \cap t$ of $P$ is said to be {\em of order $M_{st}$}\index{ridge!order}. The first objects we obtain are the Coxeter system $(S,M)$ and the Coxeter group $W$. 
}
\\
\par{
We now build an orbifold whose fundamental group is $W$ and whose underlying topological space is the starting polytope $P$ minus some faces: For each proper face $f$ of $P$, let $S_f =\{s \in S \,|\, f \subset s \}$. If $W_f := W_{S_f}$ is an infinite Coxeter group then the face $f$ is said to be {\em undesirable}\index{undesirable face}. Let $\hat{P}$ be the orbifold obtained from $P$ with undesirable faces removed, with facets as mirrors, with the remaining ridges $s \cap t$ as corner reflectors of orders $M_{st}$. We call $\hat{P}$ a {\em Coxeter $d$-orbifold}\index{Coxeter!orbifold}\index{orbifold!Coxeter}. 
We remark that a Coxeter $d$-orbifold is closed if and only if for each vertex $v$ of $P$, the Coxeter group $W_v$ is finite.
}
\\
\par{
For example, let $P$ be a polytope in $\mathbb{X}=\mathbb{S}^d, \mathbb{E}^d$ or $\mathbb{H}^d$ with dihedral angles submultiples of $\pi$. The uniqueness of the reflection across a hyperplane of $\mathbb{X}$ allows us to obtain a Coxeter $(\mathrm{Isom}(\mathbb{X}),\mathbb{X})$-orbifold $\hat P$ from $P$.
}

\subsubsection{Deformation spaces}

Recall that $\C(\hat{P})$ denotes the deformation space of properly convex real projective structures on the Coxeter orbifold $\hat{P}$, that is, the space of projective structures on $\hat{P}$ whose developing map is a diffeomorphism onto a properly convex subset in $\mathbb{RP}^d$.

\subsection{Vinberg's breakthrough}

In this subsection we give a description of Vinberg's results in his article \cite{MR0302779}. An alternative treatment is given in Benoist's notes \cite{MR2655311}.

\subsubsection{Groundwork}
\par{
Let $V$ be the real vector space of dimension $d+1$. A {\em projective reflection}\index{projective!reflection} (or simply, {\em reflection}\index{reflection}) $\sigma$ is an element of order $2$ of $\mathrm{SL}^{\pm}(V)$ which is the identity on a hyperplane $H$. All reflections are of the form $ \sigma = \Idd-\alpha \otimes b $ for some linear functional $\alpha \in V^{\star}$ and some vector $b \in V$ with $\alpha(b)=2$.
Here, the kernel of $\alpha$ is the subspace $H$ of fixed points of $\sigma$ and $b$ is the eigenvector corresponding to the eigenvalue $-1$.
}
\\
\par{
Let $P$ be a $d$-polytope in $\mathbb{S}(V)$ and let $S$ be the set of facets of $P$. For each $s \in S$, choose a reflection $\sigma_s = \Idd-\alpha_s \otimes b_s$ with $\alpha_s(b_s)=2$ which fixes $s$. By making a suitable choice of signs, we assume that $P$ is defined by the inequalities $\alpha_s \leqslant 0$, $s \in S$.
Let $\Gamma \subset \mathrm{SL}^{\pm}(V)$ be the group generated by all these reflections $(\sigma_s)_{s \in S}$ and let $\mathring{P}$ be the interior of $P$. A pair $(P,(\sigma_s)_s)$ is called a {\em projective Coxeter polytope}\index{projective!Coxeter polytope} if the family $\{ \gamma \mathring{P}  \}_{\gamma \in \Gamma } $ is pairwise disjoint.
}
\\
\par{
The $|S| \times |S|$ matrix $A=(A_{st})_{s,t \in S}$, $A_{st}=\alpha_s(b_t)$, is called the {\em Cartan matrix}\index{Cartan matrix}\index{projective!Coxeter polytope!Cartan matrix} of a projective Coxeter polytope $P$. For each reflection $\sigma_s$, the linear functional $\alpha_s$ and the vector $b_s$ are defined up to transformations
\begin{displaymath}
\alpha_s \mapsto \lambda_s\alpha_s \;\; \text{and} \;\; b_s \mapsto \lambda_s^{-1}b_s  \text{ with } \lambda_s  > 0.
\end{displaymath}
Hence the Cartan matrix of $P$ is defined up to the following equivalence relation: two matrices $A$ and $B$ are {\em equivalent}\index{Cartan matrix!equivalent} if $A=\Lambda B \Lambda^{-1}$ for a diagonal matrix $\Lambda$ having positive entries. This implies that for every $s, t \in S$, the number $A_{st}A_{ts}$ is an invariant of the projective Coxeter polytope $P$.  
}

\subsubsection{Vinberg's results}
\par{
Vinberg proved that the following conditions are necessary and sufficient for $P$ to be a projective Coxeter polytope:
\begin{enumerate}
\item[(V1)] $A_{st} \leq 0$ for $s \neq t$, and $A_{st}=0$ if and only if $A_{ts}=0$.
\item[(V2)] $A_{ss}=2$; and for $s\ne t$, $A_{st}A_{ts}\geqslant 4$ or $A_{st}A_{ts}=4\cos^2 (\frac{\pi}{m_{st}}$), $m_{st} \in \mathbb{N} \smallsetminus \{0, 1\}$.
\end{enumerate}
}

\par{
The starting point of the proof is that for every two facets $s$ and $t$ of $P$, the automorphism $\sigma_s \sigma_t$ has to be conjugate to one of the following automorphisms of $V / U$ with $U = \ker(\alpha_s) \cap \ker(\alpha_t)$:
$$
\begin{bmatrix}
\lambda & 0 \\
0 & \lambda^{-1} \\
\end{bmatrix} (\lambda > 0), \quad
\begin{bmatrix}
1 & 1 \\
0 & 1 \\
\end{bmatrix} \quad \textrm{or} \quad
\begin{bmatrix}
\cos\theta & -\sin\theta \\
\sin\theta & \;\;\,\cos\theta \\
\end{bmatrix}\, (\theta= \tfrac{2\pi}{m_{st}}).
$$

In the third case we call $\sigma_s \sigma_t$ a {\em rotation}\index{rotation} of {\em angle}\index{rotation!angle} $\theta$. 
}
\\
\par{
To a projective Coxeter polytope $P$, we can associate the Coxeter matrix $M=(M_{st})_{s,t\in S}$ with the set $S$ of facets of $P$ such that $M_{st}=m_{st}$ if $\sigma_s \sigma_t$  is a rotation of angle $\frac{2\pi}{m_{st}}$, and $M_{st}=\infty$ otherwise.
Now, from the Coxeter system $(S,M)$ and the polytope $P$, we obtain the Coxeter group $W$ and projective Coxeter orbifold $\hP$. Eventually, we are also interested in the subgroup $\G$ of $\mathrm{SL}^{\pm}(V)$ generated by all the reflections across the facets of $P$.
}

\begin{theorem}[Tits \cite{MR0240238}, Vinberg \cite{MR0302779}]\label{t:tits-vin}
Let $P$ be a projective Coxeter polytope. Then the following are true\,{\rm :}
\begin{enumerate}
\item The morphism $\sigma: W \rightarrow \Gamma$ given by $\sigma(s) =
\sigma_s$ is an isomorphism.

\item The group $\Gamma$ is a discrete subgroup of $\mathrm{SL}^{\pm}(V)$.

\item The union of tiles $\Cc :=\cup_{\gamma \in \Gamma} \gamma P$ is convex.

\item The group $\Gamma$ acts properly discontinuously on $\O$, the interior of $\Cc$, hence the quotient $\O/\Gamma$ is a convex real projective Coxeter orbifold.

\item An open face $f$ of $P$ lies in $\O$ if and only if the Coxeter group $W_{f}$ is finite.
\end{enumerate}
\end{theorem}

\begin{rem}
Tits proved Theorem \ref{t:tits-vin} (without the fifth item) assuming that $P$ is a simplex and the Cartan matrix $A$ of $P$ is symmetric and $A_{st}A_{ts} \leqslant 4 $ for all facets $s,t$ of $P$. The statements can be found in Chapter 5 (Theorem 1 of Section 4 and Proposition 6 of Section 6) of \cite{MR0240238}. The final version is due to Vinberg \cite{MR0302779}.
\end{rem}

\subsection{Convex projective Coxeter $2$-orbifolds}

In the previous section, we explain the deformation space $\C(\Sigma)$ of properly convex projective structures on a closed $2$-orbifold $\Sigma$ of negative orbifold Euler characteristic (see Theorem \ref{thm:CG}). As a special case, if $\Sigma$ is a closed projective Coxeter $2$-orbifold, then the underlying space of $\Sigma$ is a polygon and $\Sigma$ does not contain cone-points. Let $v_+$ be the number of corner reflectors of order greater than $2$, and let $\mathcal{T}(\Sigma)$ be the Teichm\"{u}ller space of $\Sigma$. 
Goldman  \cite{Goldman1977} showed that  $\C(\Sigma)$ is homeomorphic to an open cell of dimension
$$- 8 + 3v - (v-v_+)  = v_+ -2 + 2 (v-3) = v_+ - 2 + 2 \textrm{ dim } \mathcal{T}(\Sigma).$$

\subsection{Hyperbolic Coxeter $3$-orbifolds}
\par{
The Coxeter $3$-orbifolds which admit a finite-volume hyperbolic structure have been classified by Andreev \cite{MR0259734,MR0273510}.
}
\\
\par{
A polytope is naturally a CW complex. A CW complex arising from a polytope is called a {\em combinatorial polytope}\index{combinatorial polytope}\index{polytope!combinatorial}. We abbreviate a $3$-dimensional polytope to a {\em polyhedron}\index{polyhedron}. Let $\mathcal{G}$ be a combinatorial polyhedron and $(\partial \mathcal{G})^*$ be the dual CW complex of the boundary $\partial \mathcal{G}$. A simple closed curve $\gamma$ is called a {\em $k$-circuit}\index{circuit} if it consists of $k$ edges of $(\partial \mathcal{G})^*$.
A circuit $\gamma$ is {\em prismatic}\index{circuit!prismatic}\index{prismatic circuit} if all the edges of $\mathcal{G}$ intersecting $\gamma$ are disjoint.
}

\begin{theorem}[Andreev \cite{MR0259734,MR0273510}]
Let $\mathcal{G}$ be a combinatorial polyhedron, and let $\{ s_i \}_{i=1}^{n}$ be the set of facets of $\mathcal{G}$. Suppose that $\mathcal{G}$ is {\em not a tetrahedron} and non-obtuse angles $\theta_{ij} \in (0,\tfrac{\pi}{2}]$ are given at each edge $s_{ij}=s_i \cap s_j$ of $\mathcal{G}$. Then the following conditions {\rm (A1)--(A4)} are necessary and sufficient for the existence of a {\em compact} hyperbolic polyhedron $P$ which realizes\footnote{There is an isomorphism $\phi : \mathcal{G} \rightarrow P$ such that the given angle at each edge $e$ of $\Gc$ is the dihedral angle at the edge $\phi(e)$ of $P$.} $\mathcal{G}$ with dihedral angle $\theta_{ij}$ at each edge $s_{ij}$. 
\begin{enumerate}
\item[(A1)] If $s_i \cap s_j \cap s_k$ is a vertex of $\mathcal{G}$, then $ \theta_{ij} + \theta_{jk} + \theta_{ki} > \pi.$
\item[(A2)] If $s_i$, $s_j$, $s_k$ form a prismatic $3$-circuit, then $ \theta_{ij} + \theta_{jk} + \theta_{ki} < \pi.$
\item[(A3)] If $s_i$, $s_j$, $s_k$, $s_l$ form a prismatic $4$-circuit, then $ \theta_{ij} + \theta_{jk} + \theta_{kl} + \theta_{li} < 2\pi.$
\item[(A4)] If $\mathcal{G}$ is a triangular prism with triangular facets $s_1$ and $s_2$, then
        $$ \theta_{13} + \theta_{14} + \theta_{15} + \theta_{23} + \theta_{24} + \theta_{25} < 3\pi. $$
\end{enumerate}

\par{
The following conditions {\rm (F1)--(F6)} are necessary and sufficient for the existence of a {\em finite-volume} hyperbolic polyhedron $P$ which realizes $\mathcal{G}$ with dihedral angle $\theta_{ij} \in (0,\frac{\pi}{2}]$ at each edge $s_{ij}$.
}
\begin{enumerate}
\item[(F1)] If $s_i \cap s_j \cap s_k$ is a vertex of $\mathcal{G}$, then $ \theta_{ij} + \theta_{jk} + \theta_{ki} \geqslant \pi.$
\item[(F2)] {\rm (}resp. {\rm(F3)} or {\rm (F4))} is the same as {\rm (A2)} {\rm (}resp. {\rm (A3)} or {\rm (A4)).}
\item[(F5)] If $s_i \cap s_j \cap s_k \cap s_l$ is a vertex of $\mathcal{G}$, then $ \theta_{ij}+\theta_{jk}+\theta_{kl}+\theta_{li}=2\pi.$
\item[(F6)] If $s_i$, $s_j$, $s_k$ are facets such that $s_i$ and $s_j$ are adjacent, $s_j$ and $s_k$ are adjacent, and $s_i$ and $s_k$ are not adjacent
but meet at a vertex not in $s_j$, then $ \theta_{ij}+\theta_{jk} < \pi.$
\end{enumerate}
In both cases, the hyperbolic polyhedron is unique up to hyperbolic isometries.
\end{theorem}

\subsection{Convex projective Coxeter $3$-orbifolds}

\subsubsection{Restricted deformation spaces}

\par{
A point of $\C(\hat{P})$ gives us a projective Coxeter polytope $(P_0,(\sigma_s)_s)$, well defined up to projective automorphisms. We can focus on the subspace of $\C(\hat{P})$ with a {\em projectively fixed} underlying polytope $P_0$. This subspace is called the {\em restricted deformation space}\index{deformation!space!restricted} of $\hat{P}$ and denoted by $\C_{P_0}(\hat{P})$.
}
\\
\par{
Let $\hat{P}$ be a Coxeter $3$-orbifold. We now give a combinatorial hypothesis on $\hat{P}$, called the “orderability”, which allows us to say something about the restricted deformation space $\C_{P_0}(\hat{P)}$ of $\hat{P}$. A Coxeter $3$-orbifold $\hat{P}$ is {\em orderable}\index{Coxeter!orbifold!orderable}\index{orderable Coxeter orbifold} if the facets of $\hat{P}$ can be ordered so that each facet contains at most three edges that are edges of order 2 or edges in a facet of higher index.
}
\\
\par{
Let $e$ (resp. $f$, $e_2$) be the number of edges (resp. facets, edges of order 2) of $P$, and let $k(P)$ be the dimension of the group of projective automorphisms of $P$. Note that $k(P)= 3$ if $P$ is tetrahedron, $k(P) =1$ if $P$ is the cone over a polygon other than a triangle, and $k(P)=0$ otherwise.

}

\begin{theorem}[Choi \cite{Choipoly}]\label{thm:choi}
Let $\hat{P}$ be a Coxeter $3$-orbifold such that $\C(\hat{P}) \neq \varnothing$. Suppose that $\hat{P}$ is orderable and that the Coxeter group $\pi_{1}^{orb}(\hat{P})$ is infinite and irreducible. Then every restricted deformation space $\C_{P}(\hat{P})$ is a smooth manifold of dimension $3f-e-e_2-k(P)$.
\end{theorem}

A simplicial polyhedron\footnote{A {\em simplicial} polyhedron\index{polyhedron!simplicial}\index{simplicial polyhedron} is a polyhedron whose facets are all triangles.} is orderable. By Andreev's theorem, hyperbolic triangular prisms are orderable. 
However the cube and the dodecahedron do not carry any orderable Coxeter orbifold structure, since the lowest index facet in an orderable orbifold must be triangular.

\subsubsection{Truncation polyhedra}
\par{
Andreev's theorem gives the necessary and sufficient conditions for the existence of a closed or finite-volume hyperbolic Coxeter $3$-orbifold. We can think of analogous questions for closed or finite-volume properly convex projective Coxeter orbifolds.
}
\\
\par{
The third author \cite{ecima_ludo} completely answered the question of whether or not a Coxeter $3$-orbifold $\hat{P}$ admits a  convex projective structure assuming that the underlying space $P$ is a truncation polyehedron: a {\em truncation $d$--polytope}\index{polytope!truncation}\index{truncation polytope} is a $d$-polytope obtained from the $d$-simplex by iterated truncations of vertices. For example, a triangular prism is a truncation polyhedron. However the cube and the dodecahedron are not truncation polyhedra.
}
\\
\par{
A prismatic $3$-circuit of $\hat{P}$ formed by the facets $r$, $s$, $t$ is {\em bad}\index{circuit!prismatic!bad} if 
$$ \frac{1}{M_{rs}} + \frac{1}{M_{st}} + \frac{1}{M_{tr}} \geqslant 1 \quad \textrm{and} \quad 2 \in \{ M_{rs}, M_{st}, M_{tr} \}.$$ Let $e_+$ be the number of edges of order greater than $2$ in $\hat{P}$.
}

\begin{theorem}[Marquis \cite{ecima_ludo}]\label{thm:Marquis}
Let $\hat{P}$ be a Coxeter $3$-orbifold arising from a truncation polyhedron $P$. Assume that $\hP$ has no bad prismatic $3$-circuits.
If $\hat{P}$ is not a triangular prism and $e_+ > 3$, then $\C(\hat{P)}$ is homeomorphic to a finite union of open cells of dimension $e_+ -3$. 
Moreover, if $\hat{P}$ admits a hyperbolic structure, then $\C(\hat{P)}$ is connected.
\end{theorem}

The third author actually provided an explicit homeomorphism between $\C(\hat{P)}$ and the union of $q$ copies of $\mathbb{R}^{e_+ -3}$ when $\hP$ is a Coxeter truncation $3$-orbifold. Moreover, the integer $q$ can be computed easily in terms of the combinatorics and the edge orders.

\subsection{Near the hyperbolic structure}

\subsubsection{Restricted deformation spaces}

The first and second authors and Hodgson \cite{CHL} described the local restricted deformation space for a class of
Coxeter orbifolds arising from  {\em ideal}\index{polyhedron!ideal hyperbolic} hyperbolic polyhedra, i.e. polyhedra with all vertices on $\partial \mathbb{H}^3$. Note that a finite volume hyperbolic Coxeter obifold is unique up to hyperbolic isometries by Andreev’s theorem \cite{MR0259734,MR0273510} or Mostow-Prasad rigidity theorem \cite{mostow,prasad}.

\begin{theorem}[Choi-Hodgson-Lee \cite{CHL}]
Let $P$ be an ideal hyperbolic polyhedron whose dihedral angles are all equal to $\tfrac{\pi}{3}$. If $P$ is not a tetrahedron, then 
at the hyperbolic structure the restricted deformation space $\C_P(\hP)$ is smooth and of dimension $6$.
\end{theorem}

\subsubsection{Weakly orderable Coxeter orbifolds}

The first and second authors \cite{CL15} found a \lq\lq large\rq\rq\, class of Coxeter $3$-orbifolds whose local deformation spaces are understandable. A Coxeter $3$-orbifold $\hat{P}$ is {\em weakly orderable}\index{Coxeter!orbifold!weakly orderable}\index{weakly orderable Coxeter orbifold} if the facets of $P$ can be ordered so that each facet contains at most $3$ edges of order $2$ in a facet of higher index. Note that Greene \cite{MR3322033} gave an alternative (cohomological) proof of the following theorem. 

\begin{theorem}[Choi-Lee \cite{CL15}, Greene \cite{MR3322033}]\label{thm:weakly}
Let $\hP$ is a closed hyperbolic Coxeter $3$-orbifold. If $\hP$ is weakly orderable, then at the hyperbolic structure $\C(\hP)$ is smooth and of dimension $e_+ - 3$.
\end{theorem}

For example, if $P$ is a truncation polyhedron, then $\hP$ is always weakly orderable.
The cube is not a truncation polyhedron, but every closed hyperbolic Coxeter $3$-orbifold arising from the cube is weakly orderable. 
On the other hand, there exist closed hyperbolic Coxeter $3$-orbifolds arising from the dodecahedron which are not weakly orderable.
However, almost all the closed hyperbolic Coxeter $3$-orbifolds arising from the dodecahedron are weakly orderable:

\begin{theorem}[Choi-Lee \cite{CL15}]
Let $P$ be a simple\footnote{A polyhedron $P$ is {\em simple}\index{polyhedron!simple}\index{simple polyhedron} if each vertex of $P$ is adjacent to exactly three edges.} polyhedron. Suppose that $P$ has no prismatic $3$-circuit and has at most one prismatic $4$-circuit.
Then 
\begin{equation*}
\lim_{m \rightarrow \infty} \frac{ | \{ \text{weakly orderable, closed hyperbolic Coxeter $3$-orbifolds $\hP$ with
edge order $\leq m$} \} | }{
| \{ \text{closed hyperbolic Coxeter $3$-orbifolds $\hP$  with
edge order $\leq m$} \}|  } = 1
\end{equation*}
\end{theorem}

A result similar to Theorem \ref{thm:weakly} is true for higher dimensional closed Coxeter orbifolds $\hP$ whose underlying polytope $P$ is a {\em truncation polytope}\index{truncation polytope}\index{polytope!truncation}:

\begin{theorem}[Choi-Lee \cite{CL15}, Greene \cite{MR3322033}]\label{thm:truncation}
If $\hat{P}$ be a closed hyperbolic Coxeter orbifold arising from a truncation polytope $P$, then at the hyperbolic structure $\C(\hP)$ is smooth and of dimension $e_+ - d$.
\end{theorem}

\par{
We remark that if $\hP$ is not weakly orderable, then Theorem \ref{thm:weakly} is not true anymore: Let $m$ be a fixed integer greater than $3$. Consider the compact hyperbolic Coxeter polyhedron $P_1$ shown in Figure \ref{fig:counterex} (A).
Here, if an edge is labelled $m$, then its dihedral angle is $\tfrac{\pi}{m}$. Otherwise, its dihedral angle is $\tfrac{\pi}{2}$.
Let $\hP_1$ be the corresponding hyperbolic Coxeter $3$-orbifold. Then $e_+ -3=0$, but $\C(\hP_1)=\R$ (see Choi-Lee \cite{CL15}).
Of course $\hP_1$ is not weakly orderable, since every facet in $\hP_1$ contains four edges of order $2$.
}
\\
\par{
There is also a compact hyperbolic Coxeter $4$-polytope $P_2$ such that $\C(\hP_2)$ is not homeomorphic to a manifold. 
The underlying polytope $P_2$  is the product of two triangles and the Coxeter graph of $\hP_2$ is shown in Figure \ref{fig:counterex} (B).
}

\begin{figure}[ht!]
\labellist
\small\hair 2pt
\pinlabel $m$ at 245 308
\pinlabel $m$ at 245 168
\pinlabel $m$ at 245 23
\endlabellist
\centering
\subfloat[]{\includegraphics[height=25mm]{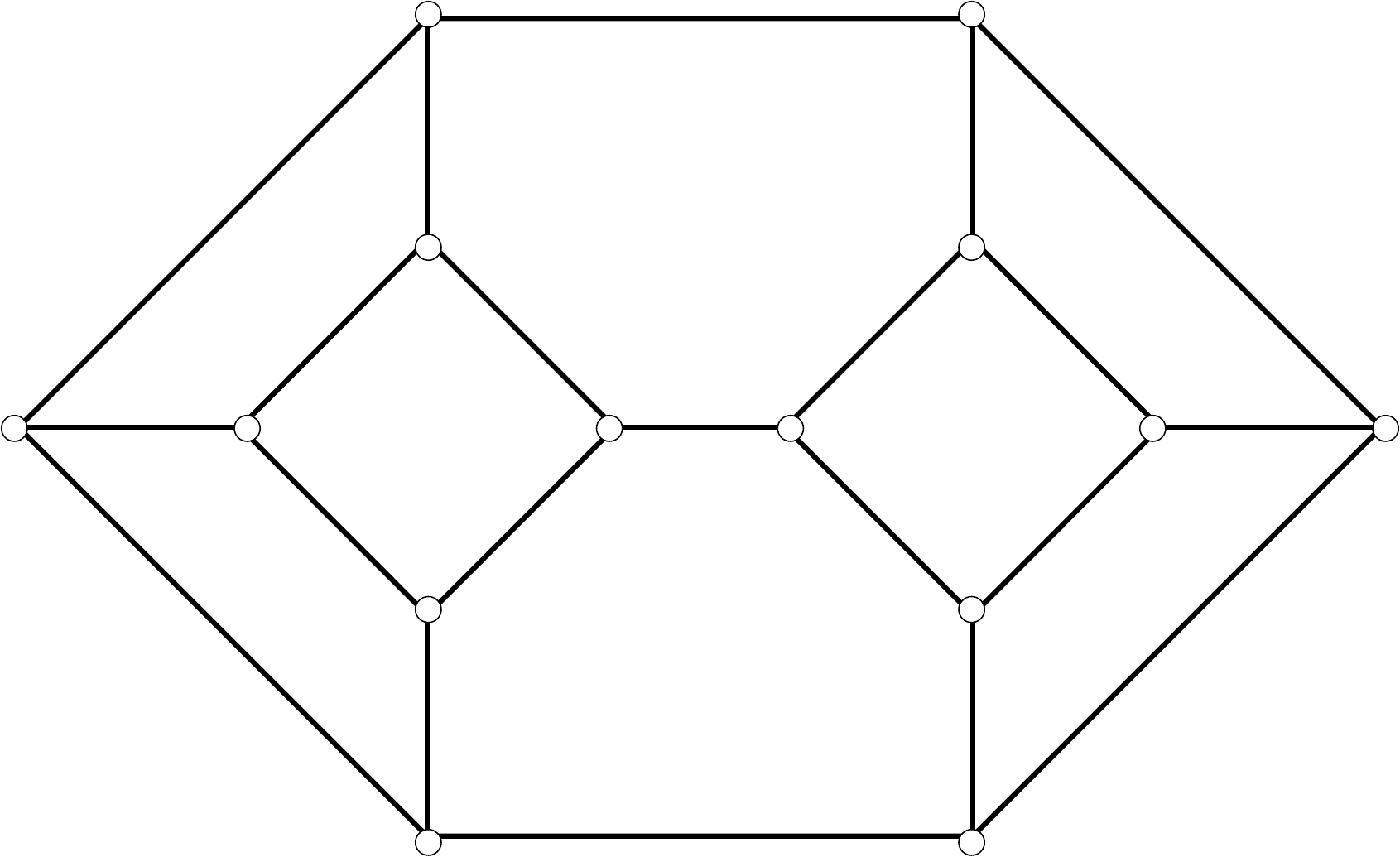}}
\quad\quad\quad\quad\quad\quad
\labellist
\small\hair 2pt
\pinlabel $5$ at 100 270
\pinlabel $5$ at 100 80
\pinlabel $5$ at 540 180
\endlabellist
\subfloat[]{\includegraphics[height=25mm]{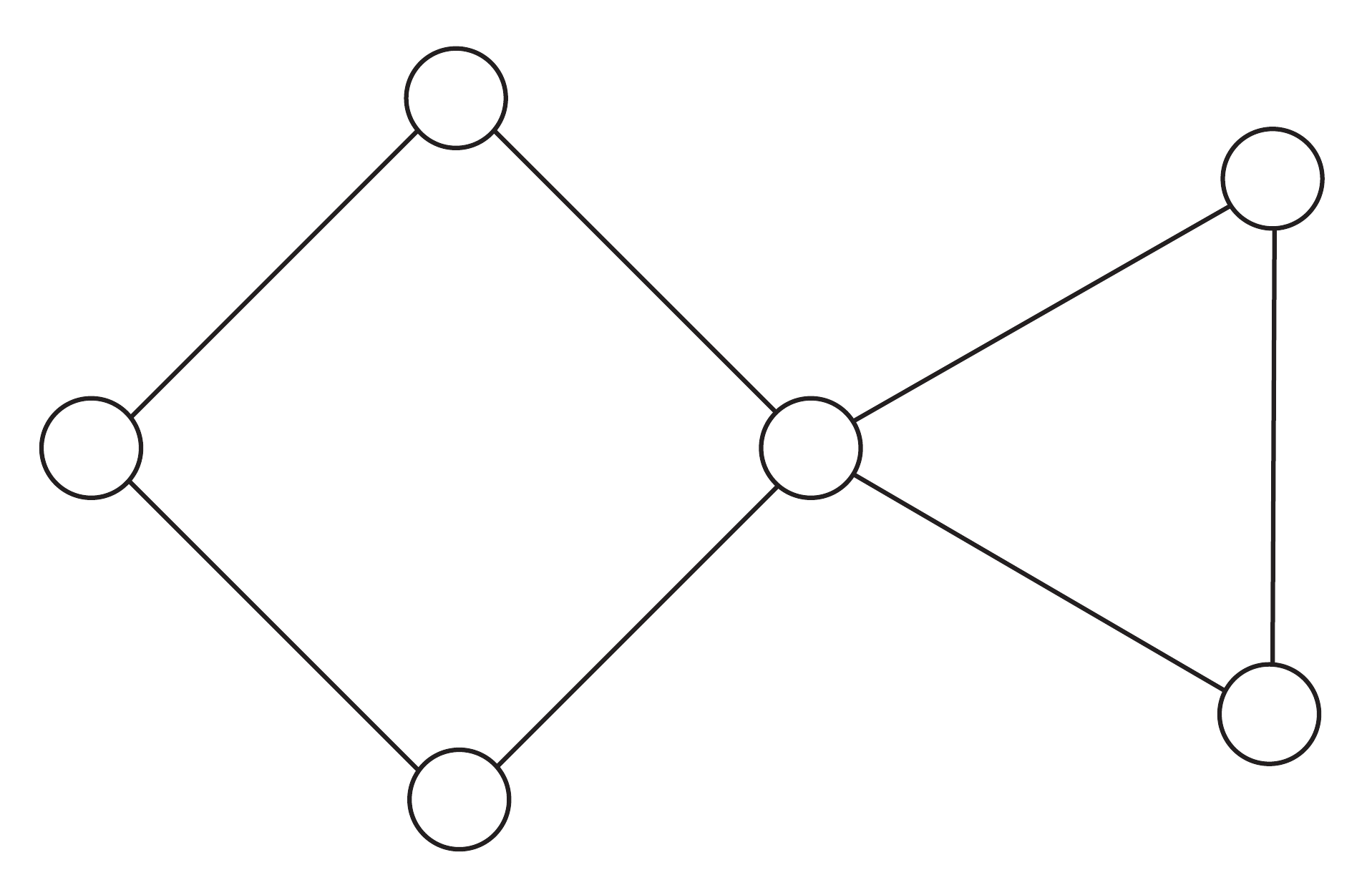}}
\caption{(A) Coxeter $3$-polytope $P_1$ \quad (B) Coxeter $4$-polytope $P_2$ }\label{fig:counterex}
\end{figure}
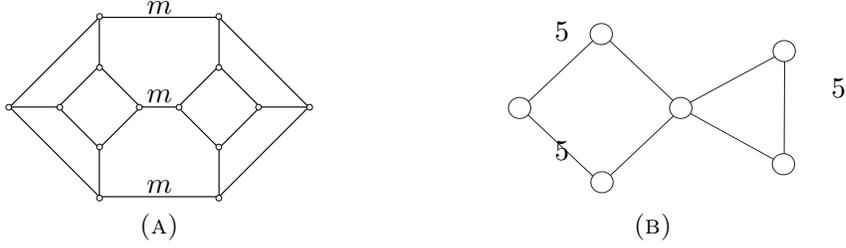

The space $\C(\hP_2)$  is homeomorphic to the following solution space (see Choi-Lee \cite{CL15}):
\begin{displaymath}
\mathcal{S} = \{(x,y) \in (\mathbb{R}_+)^2 \,\, |\,\, 8x-(5+\sqrt{5})y-(6-2\sqrt{5})xy-(5+\sqrt{5})x^2y+8xy^2 = 0 \},
\end{displaymath}
which is pictured in Figure \ref{fig:notmanifold}, and hence $\C(\hP_2)$ is not a manifold. Here the singular point $(1,1) \in \mathcal{S}$ corresponds to the hyperbolic structure in $\C(\hP_2)$.

\begin{figure}[ht]
\centering
\includegraphics[height=4cm]{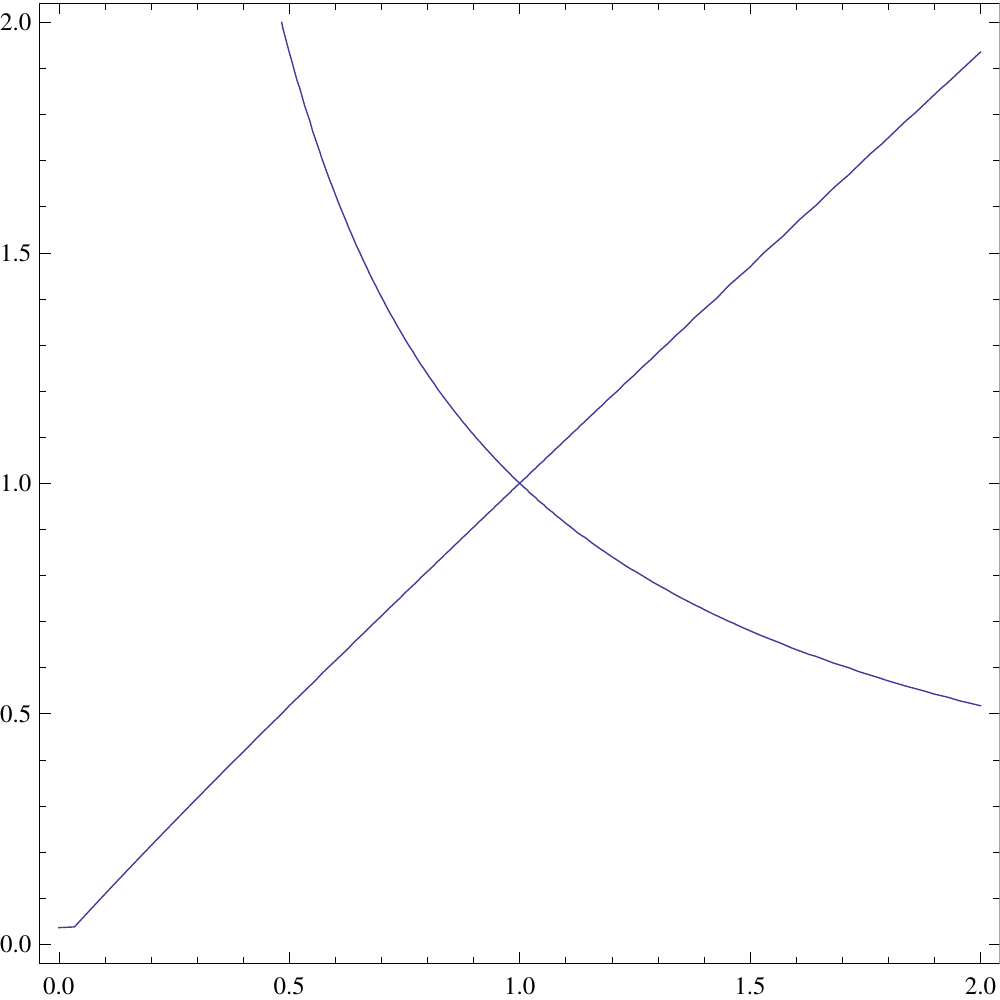}
\caption{$8x-(5+\sqrt{5})y-(6-2\sqrt{5})xy-(5+\sqrt{5})x^2y+8xy^2 = 0$}
\label{fig:notmanifold}
\end{figure}

\section{Infinitesimal deformations}

\subsection{Rigidity or deformability}

We first discuss some general theory.

\begin{de}
A representation $\rho:\Gamma \to G$ is {\em locally rigid}\index{rigid!locally}\index{locally!rigid}\index{representation!locall rigid} if the $G$-orbit of $\rho$ in $\Hom(\Gamma, G)$ contains a neighborhood of $\rho$
in $\Hom(\Gamma, G)$. Otherwise, $\rho$ is {\em locally deformable}\index{representation!locally deformable}\index{locally!deformable}.
\end{de}

If $\rho$ is locally deformable, then there exists a sequence of representations $\rho_n:\Gamma \to G$ converging to $\rho$ such that $\rho_n$ is not conjugate to $\rho$. We emphasise that $\rho_n$ has no reason to be discrete or faithful even if $\rho$ is so.

\begin{de}
Two representations $\rho, \rho': \Gamma \to G$ are of the {\em same type}\index{representation!same type} if 
for all $\g \in \G$, $\rho(\g)$ and $\rho'(\g)$ have the same type in the Jordan decomposition. 
A discrete faithful representation $\rho:\Gamma \to G$ is {\em globally rigid}\index{rigid!globally}\index{globally rigid}\index{representation!globally rigid}
if every discrete faithful representation in $\Hom(\Gamma, G)$ whose type is the same as $\rho$ is conjugate to $\rho$.
\end{de}
For example, two representations $\rho, \rho' : \Gamma \ra \mathrm{PGL}(2,\mathbb{R})$ are of the same type if and only if for each $\g \in \Gamma$, $\rho(\g)$ and $\rho'(\g)$ are both hyperbolic, parabolic or elliptic.

\subsection{What is an infinitesimal deformation?}

\par{
In this subsection, we explore the tangent space to a representation. In order to do that, we will combine differential geometry with algebraic geometry. Given a semi-algebraic set $X$ and a point $x\in X$, we say that $X$ is a {\em smooth manifold of dimension $r$ at $x$} if there is an open neighborhood $\U$ of $x$ such that the subset $\U\cap X$ is a smooth $r$-manifold. Such a point $x$ is said to be {\em smooth}\index{point!smooth}\index{smooth point}. 
}
\\
\par{
Assume now that $X$ is an algebraic set. We can define the Zariski tangent space at any point $x \in X$. If $X$ is a smooth manifold of dimension $r$ at $x$, then the Zariski tangent space at $x$ is of dimension at least $r$. Conversely, if the Zariski tangent space at $x$ is of dimension $r$ and there is a smooth $r$-manifold in $X$ containing $x$, then $X$ is a smooth manifold of dimension $r$ at $x$. A point $x \in X$ is said to be {\em singular}\index{singular!point}\index{point!singular} if there is a Zariski tangent vector which is not tangent to a smooth curve in $X$.
}

\subsection{First order}

\par{
Assume that $\rho_t$ is a smooth path in $\Hom(\G,G)$, i.e. for each $\gamma \in \Gamma$, a path $\rho_t(\gamma)$ in $G$ is smooth. Then there exists a map $u_1:\G \to \gg$ such that
$$
\forall \g\in \G,\quad  \rho_t(\g) = \exp(t\,u_1(\g) + O(t^2))\, \rho_0(\g).
$$
Since $\rho_t$ is a homomorphism, i.e. $\rho_t(\g \,\delta) = \rho_t(\g) \rho_t(\delta)$, it follows that $u_1$ is a 1-cocycle. Conversely, if $u_1 : \G \to \gg$ is a 1-cocyle, then $\rho_t$ is a homomorphism up to first order. This computation motivates the following: Given a representation $\rho:\G \to G$, we define the space of 1-cocycles $\G \to \gg$:
$$
Z^1(\G,\gg)_{\rho} = \{ u_1:\G \to \gg \, |\, u_1(\g  \delta) = u_1(\g) +  Ad_{\rho(\g)} u_1(\delta), \, \forall \g,\delta \in \G  \}.
$$
}
\par{
Moreover, since the Zariski tangent space to an algebraic variety is the space of germs of paths satisfying the equations up to first order, the Zariski tangent space $T^{Zar}_{\rho} \Hom(\G,G)$ to $\Hom(\G,G)$ at $\rho$ can be identified with the space of 1-cocyles $Z^1(\G,\gg)_{\rho}$ via the following:
$$   
\forall \g\in \G,\quad  \frac{d}{dt} \rho_t(\g) \rho_0(\g^{-1}) |_{t=0} = u_1 (\g).
$$
}
\\
\par{
Eventually we want to understand the tangent space to the character variety, hence we need to figure out which cocycles come from the conjugation. We introduce the space of 1-coboundaries:
$$
B^1(\G,\gg)_{\rho} = \{ v_1 :\G \to \gg \, |\, \exists\, u_0 \in \gg \textrm{ such that } v_1(\g) = Ad_{\rho(\g)} u_0 - u_0   \}.
$$
Every coboundary $v_1$, in fact, is tangent to the conjugation path $\rho_t = \exp(-t u_0) \, \rho_0 \, \exp(t u_0)$. The first cohomology group with coefficients in $\gg$ twisted by the adjoint action of $\rho$ is
$$
H^1(\G,\gg)_{\rho} = Z^1(\G,\gg)_{\rho}/B^1(\G,\gg)_{\rho}.
$$
}
\par{
Basically, we explain that the map 
$T^{Zar}_{\rho} \Hom(\G,G) \to Z^1(\G,\gg)_{\rho}$ is an isomorphism. In addition, under this isomorphism, the Zariski tangent vectors coming from the
$G$-conjugation of $\rho$ exactly correspond to the coboundaries.
}

\begin{de}
A representation $\rho:\G \to G$ is {\em infinitesimally rigid}\index{representation!infinitesimally rigid}\index{rigid!infinitesimally}\index{infinitesimally rigid} if $H^1(\G,\gg)_{\rho} = 0$.
\end{de}

The following theorem motivates the terminology.

\begin{theorem}[Weil's rigidity theorem \cite{weil}]\label{thm:Rigidity}
If $\rho$ is infinitesimally rigid, then $\rho$ is locally rigid.
\end{theorem}

A nice presentation of Theorem \ref{thm:Rigidity} can be found in Besson \cite{besson}. Weil, Garland and Raghunathan also computed the group $H^1(\G,\gg)_{\rho}$ in a number of important cases and showed that it is often trivial.

\begin{theorem}[Weil \cite{weil_discrete_II}, Garland-Raghunathan \cite{garland_raghunathan}, Raghunathan \cite{raghunathan_weil_H_any_Gamma_uniform}]$\,$\\
Suppose that $G$ is a semi-simple group without compact factor and $\G$ is an irreducible lattice of $G$. Denote by $i:\G \to G$ the canonical representation and let $\rho:G \to H$ be a non-trivial irreducible representation of $G$ into a semi-simple group $H$.
\begin{itemize}
\item If $H^1(\G,\gg)_{i} \neq 0$, then 
\begin{itemize}
\item either $\mathfrak{g}=\mathfrak{so}_{2,1}(\R) = \mathfrak{su}_{1,1}=\mathfrak{sl}_2(\R)$,
\item or $\mathfrak{g}=\mathfrak{so}_{3,1}(\R) =\mathfrak{sl}_2(\Cb)$ and $\G$ is a {\em non}-uniform lattice.
\end{itemize}
\item If $H^1(\G,\mathfrak{h})_{\rho \circ i} \neq 0$ and $\G$ is a {\em uniform} lattice, then $\gg = \mathfrak{so}_{d,1}(\R)$ or $\mathfrak{su}_{d,1}$. Moreover, if we write $\mathfrak{h} = V_1 \oplus \cdots \oplus V_r$ the decomposition of the $\gg$-semi-simple module $\mathfrak{h}$ into simple modules, then the highest weight of each $V_i$ is a multiple of the highest weight of the standard representation.
\end{itemize}
\end{theorem}

\subsection{Higher order}
\par{
Given a 1-cocycle $u_1\in Z^1(\G,\gg)_{\rho}$, i.e. a Zariski tangent vector to the representation variety, we may ask if $u_1$ is {\em integrable}\index{integrable}, i.e. the tangent vector to a smooth deformation. We can start with the simplest investigation: Is the 1-cocycle $u_1$ {\em integrable up to second order}, i.e. the tangent vector to a smooth deformation up to order $2$? 
Writing the expression
$$
\forall \g\in \G, \quad \rho_t(x) = \exp(t\,u_1(\g) + t^2 \,u_2(\g) + O(t^3))\, \rho_0(\g)
$$
and using the Baker-Campbell-Hausdorff formula, we see that 
$\rho_t$ is a homomorphism up to second order if and only if
$$
\forall \g,\delta \in \G, \quad u_2(\g)-u_2(\g \delta)+ Ad_{\rho_0(\g)} u_2(\delta) = \frac12 [Ad_{\rho_0(\g)} u_1(\delta), u_1(\g)] : =\frac12 [u_1,u_1](\g,\delta).
$$
}
\par{
Hence, the 1-cocycle $u_1$ is  integrable up to second order if and only if the 2-cocycle $[u_1,u_1] \in Z^2(\G,\gg)$ is a 2-coboundary. We could ask the same question for the third order and so on. We would find a sequence of obstructions, which are all in $H^2(\G,\gg)_{\rho}$. 
In other words, for each $n \geqslant 2$, if we let 
\[V_n:=\{ u_1 \in Z^1(\G,\gg)_{\rho} \,\,|\,\, o_k(u_1) = 0  \hbox{ for every } k=2, \dotsc, n-1\},\] 
then there exists a map $o_n: V_n \to H^2(\G,\gg)_{\rho}$ such that the 1-cocycle $u_1$ is integrable up to order $n$ if and only if the obstructions $o_k(u_1) =0$ for all $k=2, \dotsc , n$.
}
\\
\par{
The story ends with a good news. Recall that $G=\Gb_{\mathbb{R}}$ and $\R \llbracket t \rrbracket$ is the ring of formal power series.  A {\em formal deformation}\index{deformation!formal}\index{formal deformation} of $\rho : \G \to G$} is a representation $\tilde{\rho}:\G \to \Gb_{\R \llbracket t \rrbracket}$ whose evaluation at $t=0$ is $\rho$. A 1-cocycle $u_1$ is, by definition, the formal tangent vector to a formal deformation (or simply {\em formally integrable}\index{integrable!formally}\index{formally integrable}) if and only if the obstructions $o_n(u_1) =0$ for all $n \geqslant 2$. A priori, this does not imply that $u_1$ is the tangent vector to a smooth deformation, but this is in fact true:
}
\begin{theorem}[Artin \cite{artin}]
If a 1-cocycle $u_1$ is formally integrable, then $u_1$ is integrable.
\end{theorem}

\subsection{Examples in hyperbolic geometry}

The world of hyperbolic geometry offers a lot of interesting behaviors. Assume that $M$ is a hyperbolic $d$-dimensional manifold with or without boundary and $\Gamma$ is the fundamental group of $M$.  

\subsubsection{Hyperbolic surfaces}

A lot is known on representations of surface groups, and the story about surface groups is different from higher dimensional manifold groups, which we will eventually concentrate on. Hence, we refer the readers to their favourite surveys on surface group representations (see, for example, Goldman \cite{Gs,GR}, Labourie \cite{lecture_labourie}, Guichard \cite{hdr_olivier}).

\subsubsection{Finite-volume hyperbolic manifolds}
\par{
If $d \geqslant 3$ and $M$ has finite volume, then the famous Mostow-Prasad rigidity theorem \cite{mostow,prasad} states that the holonomy $\rho$  of $M$ is globally rigid. This (conjugacy class of) representation is the {\em geometric representation}\index{representation!geometric} of $\G$.
}
\\
\par{
However this does not imply that $\rho$ is locally rigid. Indeed, the geometric representation might be deformed to non-faithful or non-discrete representations.\footnote{It is easy to see that every discrete and faithful representations of $\G$ are of the same type.} It is a theorem of Thurston for dimension $d=3$ and of Garland  and Raghunathan for dimension $d \geqslant 4$ that $\rho$ is locally deformable if and only if $d=3$ and $M$ has a cusp. This wonderful exception in the local rigidity of finite-volume hyperbolic manifolds is the starting point of the Thurston hyperbolic Dehn surgery theorem.
}

\begin{theorem}[Thurston \cite{Thurston:2002}, Garland-Raghunathan \cite{garland_raghunathan}]\label{thm:GR}
The holonomy representation of a finite-volume hyperbolic manifold $M$ of dimension $d \geqslant 3$ is infinitesimally rigid except if $d=3$ and $M$ is not compact. In the exceptional case, the geometric representation is a smooth point of the character variety of dimension twice the number of cusps.
\end{theorem}

Bergeron and Gelander \cite{BergeronGe}  gave an alternative proof of Theorem \ref{thm:GR} using the Mostow-Prasad rigidity.

\subsubsection{Hyperbolic manifolds with boundary}

We also wish to cite this beautiful theorem which pushes this kind of question beyond the scope of finite-volume manifolds.

\begin{theorem}[Kerckhoff-Storm \cite{KerSto}]
The holonomy representation of a compact hyperbolic manifold with totally geodesic boundary of dimension $d \geqslant 4$ is infinitesimally rigid.
\end{theorem}

\section{Infinitesimal duality to complex hyperbolic geometry}

\par{
We now return to the original interest of this survey: convex projective structures on manifolds. From the point of view of representations, our problem is to understand deformations $\rho_t:\pi_1(M) \to \s{d+1}$ from the holonomy $\rho_0 :\pi_1(M) \to \so{d}(\mathbb{R})$ of the hyperbolic structure on $M$ into representations in $\s{d+1}$.
}
\\
\par{
Suppose that $M$ is a finite-volume hyperbolic manifold of dimension $d \geqslant
 3$ and $\G$ is the fundamental group of $M$. We have seen that there exists a unique discrete faithful representation $\rho_{\mathrm{geo}}$ of $\G$ into $\so{d}(\R)$, up to conjugation. If $G$ is a Lie group and $i: \so{d}(\mathbb{R}) \to G$ is a representation, then we call the conjugacy class $[i \circ \rho_{geo}]$ the {\em hyperbolic point}\index{hyperbolic!point}\index{point!hyperbolic} of the character variety $\chi(\G,G)$ and we denote it again by $\rho_{\mathrm{geo}}$. We abuse a little bit of notation here, since we ignore $i$, but in the following $i$ will always be the canonical inclusion.
}
\\
\par{
Complex hyperbolic geometry can help us to understand local deformations into $\s{d+1}$. Indeed, complex hyperbolic geometry is  “dual” to Hilbert geometry, however, only at the hyperbolic point and at the infinitesimal level.
}
\begin{rem}\label{remark_duality_complex_hyp}
The groups $\s{d+1}$ and $\mathrm{SU}_{d,1}$ are non-compact real forms of the complex algebraic group $\mathrm{SL}_{d+1}(\mathbb{C})$ that both contains the real algebraic group $\so{d}(\R)$. Moreover, the Lie algebra $\sl{d+1}$ splits as 
\begin{equation}\label{eq:split}
\sl{d+1} = \lieso{d} \oplus \lieo
\end{equation}
where $\lieo$ is the orthogonal complement to $\lieso{d}$ in $\sl{d+1}$ with respect to the Killing form of $\sl{d+1}$, and the adjoint action of $\so{d}(\mathbb{R})$ preserves the decomposition (\ref{eq:split}). Hence to study the cohomology group $H^1(M,\sl{d+1})_{\rho}$, we just have to understand $H^1(M,\lieo)_{\rho}$, since the cohomology group $H^1(M,\lieso{d})_{\rho}$ is well known. But, since the Lie algebra $\mathfrak{su}_{d,1} = \lieso{d} \oplus i\lieo$, we can find $H^1(M,\lieo)_{\rho}$ using complex hyperbolic geometry (see Heusener-Porti \cite{HP}, Cooper-Long-Thistlethwaite \cite{CLT_flexing} for more details).
\end{rem}

Remark \ref{remark_duality_complex_hyp} evolves into the following theorem:

\begin{theorem}[Cooper-Long-Thistlethwaite \cite{CLT_flexing}]
Let $M$ be a closed hyperbolic manifold, and $\G = \pi_1(M)$. Then the hyperbolic point $\rho_{geo}$ in $\chi(\G,\s{d+1})$ is smooth if and only if the corresponding hyperbolic point in $\chi(\G,\mathrm{SU}_{d,1})$ is smooth. Moreover, in that case, the dimensions of the two character varieties at the hyperbolic point are the same.
\end{theorem}

\begin{theorem}[Guichard \cite{guichard_plonge}]\label{thm:Guichard}
Let $M$ be a closed hyperbolic manifold, and $\G = \pi_1(M)$. In $\chi(\G,\mathrm{SL}_{d+1}(\mathbb{C}))$, hence in particular in $\chi(\G,\s{d+1})$ and $\chi(\G,\mathrm{SU}_{d,1})$, the representations close to the hyperbolic point $\rho_{geo}$ are faithful and discrete.
\end{theorem} 

The local pictures of the character variety for $G=\s{d+1}$ and $\mathrm{SU}_{d,1}$ are therefore the same; however the global pictures can be very different. The work of Morgan-Shalen \cite{morgan_shalen_valutations}, Bestvina \cite{bestvina_degenerations} and Paulin \cite{paulin_topologie} shows that the space of discrete and faithful representations of $\G$ in $\mathrm{SU}_{d,1}$ is compact (if $d \geqslant 2$), but this can be false if $G=\s{d+1}$. 

\section{Convex projective structures on $3$-manifolds}

\subsection{Computing character varieties}

\par{
Cooper, Long and Thistlethwaite \cite{CLT_c} investigated the local structure of the variety $\chi(\G,\s{4})$ at the hyperbolic point $\rho := \rho_{\mathrm{geo}}$ when $\G$ is the fundamental group of one of the first $4500$ closed hyperbolic 3-manifolds with $2$-generator groups in the Hodgson-Weeks census: 
}
\begin{center}
\verb|http://www.math.uic.edu/t3m/SnapPy/censuses.html|
\end{center}

\par{
We summarize their conclusions about the character variety $\chi(\G,\s{4})$ around the hyperbolic point:
\begin{itemize}
\item 4439 points, i.e. $H^1(\G,\s{4})_{\rho}=0$.
\item 9 singular points, i.e. $H^1(\G,\s{4})_{\rho} \neq 0$ but no Zariski tangent vector is integrable.
\item 43 smooth curves.
\item 7 smooth surfaces.
\item 1 singular surface such that $H^1(\G,\s{4})_{\rho}$ is $3$-dimensional.
\item 1 singular 3-variety, which has two 3-dimensional branches meeting in a curve.
\end{itemize}
}

\par{
First, we should mention that these computations are mostly done in floating-point mode, hence this summary is a very good speculation but not a statement. Second, the authors checked rigorously their speculations on certain character varieties of this list.
}

\begin{rem}
If $\G$ is the fundamental group of a closed hyperbolic 3-manifold, then the first obstruction $o_2(u_1)=0$, for every Zariski tangent vector $u_1$ at the hyperbolic point $\rho$ (see Cooper-Long-Thistlethwaite \cite{CLT_flexing}). Indeed, first, the infinitesimal rigidity of the closed hyperbolic $3$-manifold in $\so{3}(\mathbb{R})$ implies that $u_1 \in H^1(\G,\mathfrak{o})$. Second, since $[\mathfrak{o},\mathfrak{o}] \subset \mathfrak{so}_{3,1}$, it follows that $o_2(u_1)$ is not only an element of $H^{2}(\G,\mathfrak{sl}_4)$ but also an element of $H^{2}(\G,\mathfrak{so}_{3,1})$. Finally, we know from Poincar\'e duality that $H^{2}(\G,\mathfrak{so}_{3,1}) = H^{1}(\G,\mathfrak{so}_{3,1}) = 0$.
\end{rem}

The singularities of some varieties are therefore more than quadratic, since for example the manifold “Vol3” is locally rigid even if its Zariski tangent space is one-dimensional. Compare to the result of Goldman-Millson \cite{GoldMill} that the singularity at a reductive representation is at most quadratic if $\G$ is the fundamental group of a Kähler manifold.

\subsection{Infinitesimal rigidity relative to the boundary}
\par{
Heusener and Porti made use of a relative version of the infinitesimal rigidity for finite-volume hyperbolic $3$-manifolds $M$ in order to obtain 
the infinitesimal rigidity for some Dehn fillings of $M$. 
}
\\
\par{
Let $M$ be a 3-manifold with a boundary whose interior $N$ carries a finite-volume complete hyperbolic metric, and let $\rho := \rho_{geo}$ be the holonomy representation of $N$. We say that $M$ (or $N$) is {\em infinitesimally rigid relative to the boundary\index{infinitesimally rigid!relative to the boundary}} if the map $H^1(M,G)_{\rho} \to H^1(\partial M,G)_{\rho}$ is injective. Roughly speaking, at the infinitesimal level, every deformation must change the geometry of the cusp.
}
The combination of the following two theorems shows in particular that infinitely many closed hyperbolic 3-manifolds are locally rigid in $G=\s{4}$.

\begin{theorem}[Theorem 1.4 of Heusener-Porti \cite{HP}]
Infinitely many Dehn fillings of a non-compact hyperbolic 3-manifold of finite volume which is infinitesimally rigid relative to the boundary are infinitesimally rigid.
\end{theorem}

\begin{theorem}[Heusener-Porti \cite{HP}]
There exist non-compact hyperbolic 3-manifolds of finite volume which are infinitesimally rigid relative to the boundary.
\end{theorem}

\begin{rem}
For example, the figure-eight knot complement and the Whitehead link complement are infinitesimally rigid relative to the boundary. A finite-volume non-compact hyperbolic 3-manifold which contains an embedded totally geodesic closed hypersurface is not infinitesimally rigid relative to the boundary. This raises the following question: Can we find a (topological) characterization of finite-volume hyperbolic 3-manifolds which are infinitesimally rigid relative to the boundary? 
An answer even for hyperbolic knot or link complements would already be quite nice.
\end{rem}

Surprisingly, the technique of Heusener and Porti, which is extended by Ballas, also produces deformations. A {\em slope}\index{slope} is a curve in the boundary, and a slope $\g$ is {\em rigid}\index{slope!rigid} if the map 
$H^1(M,\mathfrak{o})_{\rho} \to H^1(\g ,\mathfrak{o})_{\rho}$
is non-trivial.

\begin{theorem}[Heusener-Porti \cite{HP}, Ballas \cite{Ballas_deform}]
Infinitely many {\rm (}generalized\,{\rm )} Dehn fillings of an amphichiral knot whose longitude is a rigid slope are deformable.
\end{theorem}

\bibliographystyle{abbrv}
\bibliography{bibliography2}

\printindex

\end{document}